\documentclass[10pt]{amsart}
\usepackage{amsmath,amssymb,latexsym,esint,cite,mathrsfs}
\usepackage{verbatim,wasysym}
\usepackage[left=2.65cm,right=2.65cm,top=3.15cm,bottom=3.15cm]{geometry}

\usepackage[colorlinks=true,urlcolor=blue, citecolor=red,linkcolor=blue,
linktocpage,pdfpagelabels, bookmarksnumbered,bookmarksopen]{hyperref}
\usepackage[hyperpageref]{backref}
\usepackage[english]{babel}
\renewcommand{\epsilon}{\varepsilon}

\newcommand{\pnorm}[2][]{\if #1'' \left|#2\right|_p \else \left|#2\right|_{#1} \fi}

\usepackage{tikz,enumitem,graphicx, subfig, microtype, color}
\usepackage{epic,eepic}

\newtheorem{lemma}{Lemma}[section]
\newtheorem{proposition}{Proposition}[section]
\newtheorem{theorem}{Theorem}[section]

\newtheorem{definition}{Definition}[section]

\newcommand{\eps}{\varepsilon}

\renewcommand{\O}{{\mathcal O}}
\newcommand{\maxx}{{\rm max}}

\begin{document}

\title[Critical fractional $p$-Laplacian systems]{Critical  nonlocal systems with concave-convex powers}

\author[W.\ Chen and M.\ Squassina]{Wenjing Chen and Marco Squassina}

\address[W.\ Chen]{School of Mathematics and Statistics \newline
Southwest University,
Chongqing 400715, People's Republic of China.}
\email{wjchen@swu.edu.cn}

\address[M.\ Squassina]{Dipartimento di Matematica e Fisica \newline
Universit\`a Cattolica del Sacro Cuore, Via dei Musei 41, I-25121 Brescia, Italy}
\email{marco.squassina@dmf.unicatt.it}

\thanks{W.\ Chen was supported by the National Natural Science Foundation of China No.\ 11501468 and 
	by the Natural Science Foundation of Chongqing cstc2016jcyjA0323.\ M.\ Squassina is member of the Gruppo 
	Nazionale per l'Analisi Matematica, la Probabilit\`a e le loro Applicazioni}

\begin{abstract}
By using the fibering method jointly with Nehari manifold techniques, 
we obtain the existence of multiple solutions to a fractional $p$-Laplacian system
involving critical concave-convex nonlinearities provided that a suitable smallness condition
on the parameters involved is assumed. The result is obtained despite there is no
general classification for the optimizers of the critical fractional Sobolev embedding.
\end{abstract}

\maketitle

\newtheorem{Lemma}{Lemma}[section]
\newtheorem{Theorem}{Theorem}[section]
\newtheorem{Definition}{Definition}[section]
\newtheorem{Proposition}{Proposition}[section]
\newtheorem{Remark}{Remark}[section]
\newtheorem{Corollary}{Corollary}[section]


%

\section{Introduction}\label{intro}
In this  work, we study the multiplicity of solutions to the following fractional elliptic system
\begin{eqnarray}\label{frac1}
\left\{ \arraycolsep=1.5pt
   \begin{array}{lll}
(-\Delta)_p^s u =\lambda |u|^{q-2}u+ \frac{2\alpha}{\alpha+\beta}|u|^{\alpha-2}u|v|^{\beta}\ \   &
{\rm in}\ \Omega \\[2mm]
(-\Delta)_p^s v =\mu |v|^{q-2}v+ \frac{2\beta}{\alpha+\beta}|u|^{\alpha}|v|^{\beta-2}v\ \   &
{\rm in}\ \Omega \\[2mm]
u=v=0 \ \  & {\rm in}\ \mathbb{R}^n\setminus\Omega,
\end{array}
\right.
\end{eqnarray}
where $\Omega$ is a smooth bounded set in $\mathbb{R}^n$, $n>ps$ with $s\in(0,1)$, 
$\lambda,\mu>0$ are two parameters, $1<q<p$ and $\alpha>1,\beta>1$ satisfy
$ \alpha+\beta=p_s^\ast$, where $p_s^\ast =np/(n-ps)$ is the fractional critical 
Sobolev exponent and $(-\Delta)_p^s$ is the fractional $p$-Laplacian operator, defined on smooth functions as
\begin{eqnarray}\label{frac3}
(-\Delta)_p^su(x)=2\lim\limits_{\epsilon\to0}\int_{\mathbb{R}^n\backslash B_\epsilon(x)}\frac{|u(y)-u(x)|^{p-2}(u(y)-u(x))}{|x-y|^{n+ps}}dy,\quad\,\,\,  x\in\mathbb{R}^n.
\end{eqnarray}
This definition is consistent, up to a normalization constant depending on $n$ and $s$, with the linear
fractional Laplacian $(-\Delta)^s$, for the case $p=2$.
If we set $\alpha=\beta$, $\alpha+\beta=r$, $\lambda=\mu$ and $u=v$, 
then system \eqref{frac1} reduces to the following fractional equation with concave-convex nonlinearities
\begin{eqnarray}\label{frac1singl}
\left\{ \arraycolsep=1.5pt  
   \begin{array}{ll}
(-\Delta)_p^s u =\lambda |u|^{q-2}u+  |u|^{r-2}u\ \  \ &
{\rm in}\ \Omega \\
u=0 \ \ \quad & {\rm in}\ \mathbb{R}^n\setminus\Omega,
\end{array}
\right.
\end{eqnarray}
where $1<q<p$ and $p<r<p_s^\ast$. In \cite{gs} 
Goyal and Sreenadh studied the existence and multiplicity of non-negative solutions to problem \eqref{frac1singl}
for subcritical concave-convex nonlinearities.
For the fractional $p$-Laplacian, consider the following general problem
\begin{eqnarray*}
\quad\left\{ \arraycolsep=1.5pt
   \begin{array}{ll}
(-\Delta)_p^s u=f(x,u)\ \ \ &
{\rm in}\ \Omega \\
u=0 \ \ \quad & {\rm in}\ \mathbb{R}^n\setminus\Omega.
\end{array}
\right.
\end{eqnarray*}
So far various results have been obtained for these kind of problems.
In the works \cite{ll}, the
eigenvalue problem associated with $(-\Delta)_p^s$ is considered and some properties of
the first and of higher (variational) eigenvalues were obtained. 
Some results about the existence of solutions have been considered in \cite{gsa,llps,psy1}, see also the references therein.
On the other hand, the fractional problems  for $p=2$ have been investigated by many researchers,
see for example \cite{svmountain} for the subcritical case, \cite{bcss,servadeivaldinociBN} for the critical case.
In particular, the authors of \cite{bcp} studied the fractional Laplacian equation involving a concave-convex nonlinearity in the subcritical case.
Moreover, by Nehari manifold and fibering maps arguments, the authors of \cite{chendeng} obtained the existence of multiple solutions to \eqref{frac1singl} for both the subcritical and critical case. The existence and multiplicity of solutions for the system when $s=1$ were considered by many authors, we refer to 
\cite{hsu,hsulin,wu} and references therein. In particular, in \cite{hsu}, multiple solutions for the critical elliptic system
\begin{eqnarray*}\label{frac1ain}
\left\{ \arraycolsep=1.5pt
   \begin{array}{lll}
-\Delta_pu =\lambda |u|^{q-2}u+ \frac{2\alpha}{\alpha+\beta}|u|^{\alpha-2}u|v|^{\beta} &
{\rm in}\ \Omega\\[1mm]
-\Delta_p v =\mu |v|^{q-2}v+ \frac{2\beta}{\alpha+\beta}|u|^{\alpha}|v|^{\beta-2}v \ \  \ &
{\rm in}\ \Omega\\[1mm]
u=v=0 \ \ \quad & {\rm on}\ \partial\Omega,
\end{array}
\right.
\end{eqnarray*}
where $q<p$ and $\alpha>1,\beta>1$ satisfy
$ \alpha+\beta=np/(n-p)$ were obtained. For the fractional system with $p=2$, we mention \cite{squadru,hemzou}.
However, as far as we know, there are a few results on the case $p\neq 2$ with concave-convex critical nonlinearities.  Recently, 
in \cite{chendeng2} system \eqref{frac1} was studied with {\em subcritical} 
concave-convex type nonlinearity, namely when $\alpha+\beta<p^*_s$.
Motivated by above results, in the present paper, we are interested in the multiplicity of solutions 
for {\em critical} fractional  $p$-Laplacian system \eqref{frac1}, namely $$\alpha+\beta=p_s^\ast.$$
We denote by $W^{s,p}(\Omega)$ the usual fractional Sobolev space endowed with the norm
\begin{equation*}
\|u\|_{W^{s,p}(\Omega)}:=\|u\|_{L^p(\Omega)}+\left(\int_{\Omega\times\Omega}\frac{|u(x)-u(y)|^p}{|x-y|^{n+ps}}dx\,dy\right)^{1/p}.
\end{equation*}
Set $Q:=\mathbb{R}^{2n}\setminus(\mathcal{C}\Omega\times\mathcal{C}\Omega)$ with 
$\mathcal{C}\Omega=\mathbb{R}^n\setminus\Omega$. We define
\begin{equation*}
X:=  \left\{u:\mathbb{R}^n\to\mathbb{R}\ \mbox{measurable},\ u|_\Omega\in L^p(\Omega) \ \mbox{and}\  \int_Q\frac{|u(x)-u(y)|^p}{|x-y|^{n+ps}}dx\,dy<\infty\right\}.
\end{equation*}
The space $X$ is endowed with the following norm 
\begin{equation*}
\|u\|_X :=\|u\|_{L^p(\Omega)}+\left(\int_Q\frac{|u(x)-u(y)|^p}{|x-y|^{n+ps}}dx\,dy\right)^{1/p}.
\end{equation*}
The space $X_0$ is defined as $X_0:=\{u\in X: \text{$u=0$ on $\mathcal{C}\Omega$}\}$ or equivalently as
$\overline{C_0^\infty(\Omega)}^X\!\!\!$ and, for any $p>1$, it is a uniformly convex Banach space endowed with the norm defined by
\begin{equation}\label{normabad}
\|u\|_{X_0} = \left(\int_Q\frac{ |u(x)-u(y)|^{p}}{|x-y|^{n+ps}}dx\,dy\right)^{1/p}.
\end{equation}
Since $u=0$ in $\mathbb{R}^n\setminus\Omega$, the integral in 
\eqref{normabad} can be extended to all $\mathbb{R}^n$.
The embedding $X_0\hookrightarrow L^r(\Omega)$ is continuous for any $r\in [1,p_s^\ast]$ and
compact for $r\in [1,p_s^\ast)$.
We set $E:=X_0\times X_0$,  with the norm
\begin{equation*}
\|(u,v)\|  =  \left(\|u\|_{X_0}^p+\|v\|_{X_0}^p\right)^{\frac{1}{p}} =  \left(\int_Q\frac{ |u(x)-u(y)|^{p}}{|x-y|^{n+ps}}dx\,dy+\int_Q\frac{ |v(x)-v(y)|^{p}}{|x-y|^{n+ps}}dx\,dy\right)^{\frac{1}{p}}. 
\end{equation*}
For convenience, we define
\begin{eqnarray}\label{huaa}
\mathcal{A}(u,\phi):=\int_Q\frac{\big| u(x)-u(y)\big|^{p-2} \big(u(x)-u(y)\big)\big( \phi(x)-\phi(y)\big)}{|x-y|^{n+ps}} \,dx\,dy.
\end{eqnarray}

\begin{definition}\label{weaksolutions}
We say that $(u,v)\in E$ is a weak solution of problem \eqref{frac1} if 
\begin{align*}\label{frac4}
\mathcal{A}(u,\phi)+\mathcal{A}(v,\psi)=\int_\Omega \left(\lambda  |u|^{q-2}u\phi+ \mu |v|^{q-2}v\psi \right)dx
+\frac{2\alpha}{\alpha+\beta}\int_\Omega|u|^{\alpha-2}u|v|^\beta\phi dx
+\frac{2\beta}{\alpha+\beta}\int_\Omega|u|^{\alpha}|v|^{\beta-2}v\psi dx
\end{align*}
for all $(\phi,\psi)\in E$.
\end{definition}

\noindent
In the sequel we omit the term {\em weak} when referring to solutions which satisfy Definition \ref{weaksolutions}. 
\vskip3pt
\noindent
Let $s\in(0,1)$, $p>1$ and let $\Omega$ be a bounded domain of $\mathbb{R}^n$. The next is our main result.

\begin{theorem}\label{fracmain}
Assume that 
\begin{equation}
\label{conditions-main}
p^2s<n<
\begin{cases}
\infty & \text{if $p\geq 2$}, \\
\frac{ps}{2-p} & \text{if $p<2$},
\end{cases}
\qquad
\frac{n(p-1)}{n-ps}\leq q<p,
\qquad \alpha+\beta=\frac{np}{n-ps}.
\end{equation}
Then there exists a positive constant $\Lambda_\ast=\Lambda_*(p,q,s,n,|\Omega|)$ such that for 
$$
0<\lambda^{\frac{p}{p-q}}+\mu^{\frac{p}{p-q}}< \Lambda_\ast,
$$
the system \eqref{frac1} admits at least two nontrivial solutions.
\end{theorem}

\noindent
For the critical case, since the embedding $X_0\hookrightarrow L^{p_s^\ast}(\mathbb{R}^n)$ fails to be compact, the energy functional does {\em not} satisfy the Palais-Smale condition globally, but that holds true when the energy level falls inside a suitable range related to the best fractional critical Sobolev constant $S$, namely 
\begin{eqnarray}\label{criticalfrac}
S  :=\inf\limits_{u\in X_0\backslash\{0\}} \frac{\displaystyle\int_{\mathbb{R}^{2n}}\frac{|u(x)-u(y)|^p}{|x-y|^{n+ps}}dxdy}{\left(\displaystyle\int_{\Omega}|u(x)|^{\frac{np}{n-ps}}dx\right)^{\frac{n-ps}{n}}}.
\end{eqnarray}
For the critical fractional case with $p\neq 2$, the main difficulty is the lack of an explicit formula for minimizers of $S$ which is very
often a key tool to handle the estimates leading to the compactness range of the functional. It was 
conjectured that, up to a multiplicative constant,  all minimizers are of the form $U((x-x_0)/\eps),$ with  
\begin{equation*}
U(x)=(1+|x|^{\frac{p}{p-1}})^{-\frac{n-ps}{p}},\quad x\in \mathbb{R}^n.
\end{equation*}
This conjecture was proved in \cite{classif} for $p=2$, but for $p\neq 2$, it is not even 
known if these functions are minimizers of $S$. On the other hand, as in \cite{mpsy}, we can overcome 
this difficulty by the optimal asymptotic behavior of minimizers, which was recently obtained in \cite{brasco}.
This will allow us to prove Lemma \ref{psc}, related to the Palais-Smale condition. That is the only point where the restriction 
\eqref{criticalfrac} on $p,q,n$ comes into play. On the other hand we point out that, as detected in \cite{mpsy}, $n=p^2s$ 
corresponds to the critical dimension
for the nonlocal Br\'ezis-Nirenberg problem.
\smallskip

This paper is organized as follows. In Section \ref{premed}, we give some notations and preliminaries for Nehari manifold and fibering maps. In Section \ref{ps}, we show $(PS)_c$ condition holds for $J_{\lambda,\mu}$ with $c$ in certain interval. In Sections \ref{Exx} and
\ref{finalsect}, we complete the proof of Theorem \ref{fracmain}.

\section{The fibering properties}\label{premed}

In this section, we give some notations and preliminaries for the Nehari manifold and the analysis of the fibering maps.
Being a weak solution $(u,v)\in E$ is equivalent to being a critical point of the following $C^1$ functional on $E$
\begin{align*}
J_{\lambda,\mu}(u,v)&:=\frac{1}{p}\int_Q\frac{|u(x)-u(y)|^p}{|x-y|^{n+ps}} \,dx\,dy
+\frac{1}{p}\int_Q\frac{|v(x)-v(y)|^p}{|x-y|^{n+ps}} \,dx\,dy \\
&-\frac{1}{q}\int_\Omega \left(\lambda  |u|^{q}+ \mu |v|^{q} \right)dx-\frac{2}{\alpha+\beta}\int_\Omega|u|^{\alpha}|v|^\beta dx.
\end{align*}
By a direct calculation, we have that $J_{\lambda,\mu}\in C^1(E,\mathbb{R})$ and
\begin{align*}
\langle J_{\lambda,\mu}'(u,v),(\phi,\psi)\rangle 
&= \mathcal{A}(u,\phi)
+\mathcal{A}(v,\psi)-\int_\Omega \left(\lambda  |u|^{q-2}u\phi+ \mu |v|^{q-2}v\psi \right)dx\nonumber\\
&-\frac{2\alpha}{\alpha+\beta}\int_\Omega|u|^{\alpha-2}u|v|^\beta\phi dx
-\frac{2\beta}{\alpha+\beta}\int_\Omega|u|^{\alpha}|v|^{\beta-2}v\psi dx
\end{align*}
for any $(\phi,\psi)\in E$.
We will study critical points of the function $J_{\lambda,\mu}$ on $E$.
Consider the {\em Nehari} manifold
$$
\mathcal{N}_{\lambda,\mu}=\left\{(u,v)\in E\backslash\{(0,0)\}:  \langle J_{\lambda,\mu}'(u,v),(u,v)\rangle =0\right\}.
$$
Then, $(u,v)\in \mathcal{N}_{\lambda,\mu}$ if and only if $(u,v)\neq (0,0)$ and 
\begin{equation*}
\|(u,v)\|^p=\int_\Omega (\lambda |u|^{q}+\mu |v|^{q})dx +2\int_\Omega  |u|^{\alpha}|v|^\beta dx.
\end{equation*}
The Nehari manifold $\mathcal{N}_{\lambda,\mu}$ is closely linked to the behavior of the function of the form $\varphi_{u,v}:t\mapsto J_{\lambda,\mu}(tu,tv)$ for $t>0$ defined by
$$
\varphi_{u,v}(t):=J_{\lambda,\mu}(tu,tv)=\frac{t^p}{p}\|(u,v)\|^p- \frac{t^{q}}{q}\int_\Omega  (\lambda |u|^{q}+\mu |v|^q)dx -\frac{2t^{\alpha+\beta}}{\alpha+\beta}\int_\Omega |u|^{\alpha}|v|^\beta dx.
$$
Such maps are known as {\em fibering maps} and were introduced by Drabek and Pohozaev in \cite{dp}.

\begin{lemma}[Fibering map]
	\label{lem2}
Let $(u,v)\in E\backslash\{(0,0)\}$, then  $(tu,tv)\in \mathcal{N}_{\lambda,\mu}$ if and only if $\varphi_{u,v}'(t)=0$.
\end{lemma}
\begin{proof}
The result is a consequence of the fact that
$
\varphi_{u,v}'(t)=\langle J_{\lambda,\mu}'(tu,tv),(u,v)\rangle.
$
\end{proof}

\noindent
We note that
\begin{equation}\label{nehar2}
\varphi_{u,v}'(t)=t^{p-1}\|(u,v)\|^p- t^{q-1}\int_\Omega  (\lambda |u|^{q}+\mu |v|^{q})dx -2t^{\alpha+\beta-1}\int_\Omega  |u|^{\alpha}|v|^\beta dx,
\end{equation}
and
\begin{equation*}
\varphi_{u,v}''(t)= (p-1)t^{p-2}\|(u,v)\|^p- (q-1)t^{q-2}\int_\Omega  (\lambda |u|^{q}+\mu |v|^{q})dx -2(\alpha+\beta-1)t^{\alpha+\beta-2}\int_\Omega  |u|^{\alpha}|v|^\beta dx.
\end{equation*}
By Lemma \ref{lem2}, $(u,v)\in \mathcal{N}_{\lambda,\mu}$ if and only if $\varphi_{u,v}'(1)=0$. Hence for $(u,v)\in \mathcal{N}_{\lambda,\mu}$, 
\eqref{nehar2} yields
\begin{align}\label{ma2}
\varphi_{u,v}''(1) &= (p-1)\|(u,v)\|^p-(q-1)\int_\Omega (\lambda |u|^{q}+\mu|v|^q)dx-2(\alpha+\beta-1) \int_\Omega  |u|^{\alpha}|v|^\beta dx\nonumber\\
&=2(p-(\alpha+\beta)) \int_\Omega |u|^{\alpha}|v|^\beta dx+(p-q)\int_\Omega (\lambda |u|^{q}+\mu|v|^q)dx\nonumber\\
&=(p-q)\|(u,v)\|^p-2((\alpha+\beta)-q)\int_\Omega |u|^{\alpha}|v|^\beta dx\nonumber\\
&= (p-(\alpha+\beta))\|(u,v)\|^p+ ((\alpha+\beta)-q)\int_\Omega (\lambda |u|^{q}+\mu|v|^q)dx.
\end{align}
Thus, it is natural to split $\mathcal{N}_{\lambda,\mu}$ into three parts corresponding to local minima, local maxima 
and points of inflection of $\varphi_{u,v}$, namely
\begin{align*}
& \mathcal{N}_{\lambda,\mu}^+:=\big\{ (u,v)\in \mathcal{N}_{\lambda,\mu} : \varphi_{u,v}''(1) >0\big\},  \\
& \mathcal{N}_{\lambda,\mu}^-:=\big\{ (u,v)\in \mathcal{N}_{\lambda,\mu} : \varphi_{u,v}''(1) <0\big\},   \\
&\mathcal{N}_{\lambda,\mu}^0:=\big\{ (u,v)\in \mathcal{N}_{\lambda,\mu} : \varphi_{u,v}''(1) =0\big\}.
\end{align*}
We will prove the existence of solutions of problem \eqref{frac1} by investigating the existence of minimizers of functional $J_{\lambda,\mu}$ on $\mathcal{N}_{\lambda,\mu}$. Although $\mathcal{N}_{\lambda,\mu}$ is a subset of $E$, we can see that the local minimizers on Nehari manifold $\mathcal{N}_{\lambda,\mu}$ are usually critical points of $J_{\lambda,\mu}$. We have the following 

\begin{lemma}[Natural constraint]
	\label{le0} 
Suppose that $(u_0,v_0)$ is a local minimizer of $J_{\lambda,\mu}$ on $\mathcal{N}_{\lambda,\mu}$ and 
that $(u_0,v_0)\not\in \mathcal{N}_{\lambda,\mu}^0$.\ Then $(u_0,v_0)$ is a critical point of $J_{\lambda,\mu}$.
\end{lemma}
\begin{proof}
The proof is a standard corollary of the lagrange multiplier rule where the constraint is 
$$
Q(u,v)=\|(u,v)\|^p-\int_\Omega (\lambda |u|^{q}+\mu |v|^{q})dx -2\int_\Omega  |u|^{\alpha}|v|^\beta dx,
$$
after observing that, for $(u,v)\in \mathcal{N}_{\lambda,\mu}$, then
\begin{align*}
\langle Q'(u,v), (u,v)\rangle&=p\|(u,v)\|^p-q\int_\Omega (\lambda |u|^{q}+\mu |v|^{q})dx -2(\alpha+\beta)\int_\Omega  |u|^{\alpha}|v|^\beta dx \\
&=(p-1)\|(u,v)\|^p-(q-1)\int_\Omega (\lambda |u|^{q}+\mu|v|^q)dx-2(\alpha+\beta-1) \int_\Omega  |u|^{\alpha}|v|^\beta dx=\varphi_{u,v}''(1)\neq 0,
\end{align*}
by the assumption that $(u,v)\not\in \mathcal{N}_{\lambda,\mu}^0$. 
\end{proof}

\noindent
In order to understand the Nehari manifold and the fibering maps, we consider $\Psi_{u,v}: \mathbb{R}^+\to \mathbb{R}$ defined by
\begin{equation*}
\Psi_{u,v}(t):=t^{p-(\alpha+\beta)}\|(u,v)\|^p-t^{q-(\alpha+\beta)}\int_\Omega (\lambda|u|^{q}+\mu|v|^q) dx.
\end{equation*}

\noindent
By simple computations, we have the following results.

\begin{lemma}[Properties of $\Psi_{u,v}$]
	\label{frib0}
Let $(u,v)\in E\setminus \{(0,0)\}$. Then $\Psi_{u,v}$ satisfies the following properties

\quad (a) $\Psi_{u,v}(t)$ has a unique critical point at 
$$
t_{\max}(u,v):=\left(\frac{(\alpha+\beta-q)\displaystyle\int_\Omega (\lambda|u|^{q}+\mu|v|^q) dx}{(\alpha+\beta-p) \|(u,v)\|^p}\right)^{\frac{1}{p-q}}>0;
$$

\quad (b) $\Psi_{u,v}(t)$ is strictly increasing on $(0,t_{\maxx}(u,v))$ and strictly decreasing on $(t_{\maxx}(u,v),+\infty)$;

\quad (c) $\lim\limits_{t\to0^+}\Psi_{u,v}(t)=-\infty$, $\lim\limits_{t\to+\infty}\Psi_{u,v}(t)=0$.
\end{lemma}

\begin{lemma}[Characterization of $\mathcal{N}_{\lambda,\mu}^\pm$]
	\label{firb1}
We have $(tu,tv)\in\mathcal{N}_{\lambda,\mu}^\pm$ if and only if $\pm \Psi_{u,v}'(t)>0$.
\end{lemma}
\begin{proof}
It is clear that for $t>0$, $(tu,tv)\in\mathcal{N}_{\lambda,\mu}$ if and only if
\begin{equation}\label{nehar40}
\Psi_{u,v}(t)=2\int_\Omega  |u|^{\alpha} |v|^\beta dx.
\end{equation}
Moreover,
\begin{equation*}
\Psi_{u,v}'(t)=(p-(\alpha+\beta))t^{p-(\alpha+\beta)-1}\|(u,v)\|^p- (q-(\alpha+\beta))t^{q-(\alpha+\beta)-1}\int_\Omega (\lambda|u|^{q}+\mu|v|^q) dx,
\end{equation*}
and if $(tu,tv)\in \mathcal{N}_{\lambda,\mu}$, then
\begin{equation}\label{nehar4ae}
t^{\alpha+\beta-1}\Psi_{u,v}'(t)=\varphi_{u,v}''(t)=t^{-2}\varphi''_{tu,tv}(1).
\end{equation}
Hence, $(tu,tv)\in\mathcal{N}_{\lambda,\mu}^+$ (resp.\ $\mathcal{N}_{\lambda,\mu}^-$) if and only if $\Psi_{u,v}'(t)>0$ (resp.\ $<0$).
\end{proof}

\begin{lemma}[Elements of $\mathcal{N}_{\lambda,\mu}^\pm$]
	\label{firb2}
	Let us set
	\begin{equation}\label{lambda1}
	\Lambda_1=\left(\frac{p-q}{2(\alpha+\beta-q)} \right)^{\frac{p}{\alpha+\beta-p}}\left(\frac{ \alpha+\beta-q}{\alpha+\beta-p}|\Omega|^{\frac{\alpha+\beta-q}{\alpha+\beta}}\right)^{-\frac{p}{p-q}}S^{\frac{\alpha+\beta}{\alpha+\beta-p}+\frac{q}{p-q}},
	\end{equation}
being $S$ the best constant for the Sobolev embedding of $X_0$ into $L^{p^*_s}({\mathbb R}^n)$.
	If $(u,v)\in E\backslash\{(0,0)\}$, then for any 
	$$
	0<\lambda^{\frac{p}{p-q}}+\mu^{\frac{p}{p-q}}<\Lambda_1,
	$$ 
	there are unique $t_1,t_2>0$ such that $t_1<t_{\max}(u,v)<t_2$ and
	$$
	(t_1 u,t_1v)\in\mathcal{N}_{\lambda,\mu}^+\quad \mbox{and}\quad (t_2 u,t_2v)\in\mathcal{N}_{\lambda,\mu}^-.
	$$
	Moreover,
	$$
	J_{\lambda,\mu}(t_1u, t_1v)=\inf\limits_{0\leq t\leq t_{\max}}J_{\lambda,\mu}(tu,tv),\qquad
	J_{\lambda,\mu}(t_2u, t_2v)=\sup\limits_{ t\geq 0}J_{\lambda,\mu}(tu,tv).
	$$
\end{lemma}
\begin{proof}
		As $\int_\Omega  |u|^{\alpha} |v|^\beta dx>0$, we know that (\ref{nehar40}) has {\em no} solution iff $\lambda$ and $\mu$ satisfy the following condition
		\begin{eqnarray*}
			2\int_\Omega  |u|^{\alpha} |v|^\beta dx > \Psi_{u,v}(t_{\max}(u,v)).
		\end{eqnarray*}
		By Lemma \ref{frib0}, we have
		\begin{align*} 
			\Psi_{u,v}(t_{\max}(u,v)) 
			&= \Big[\left(\frac{\alpha+\beta-q}{ \alpha+\beta-p }\right)^{\frac{p-(\alpha+\beta)}{p-q}}
			-\left(\frac{\alpha+\beta-q}{\alpha+\beta-p}\right)^{\frac{q-(\alpha+\beta)}{p-q}}\Big]
			\frac{\Big(\displaystyle\int_\Omega (\lambda|u|^{q}+\mu|v|^q) dx\Big)^{\frac{p-(\alpha+\beta)}{p-q}}}{\|(u,v)\|^{\frac{p(q-(\alpha+\beta))}{p-q}}}\nonumber\\
			&=  \frac{p-q}{\alpha+\beta-q}\left(\frac{\alpha+\beta-q}{ \alpha+\beta-p }\right)^{\frac{p-(\alpha+\beta)}{p-q}}\frac{\Big(\displaystyle\int_\Omega (\lambda|u|^{q}+\mu|v|^q) dx\Big)^{\frac{p-(\alpha+\beta)}{p-q}}}{\|(u,v)\|^{\frac{p(q-(\alpha+\beta))}{p-q}}}.
		\end{align*}
		By H\"{o}lder inequality and   the definition of $S$, we find
		$$
		\int_\Omega (\lambda|u|^{q}+\mu|v|^q)dx\leq S^{-\frac{q}{p}}|\Omega|^{\frac{\alpha+\beta-q}{\alpha+\beta}}\left(\lambda^{\frac{p}{p-q}}+\mu^{\frac{p}{p-q}}\right)^{\frac{p-q}{p}}\|(u,v)\|^q.
		$$
		Then, since $q<p<\alpha+\beta=p_s^\ast$, we have
		\begin{align}\label{nehar40a}
		& \Psi_{u,v}(t_{\max}(u,v))\nonumber\\
		&\geq  \frac{p-q}{\alpha+\beta-q}\left(\frac{\alpha+\beta-q}{ \alpha+\beta-p }\right)^{\frac{p-(\alpha+\beta)}{p-q}}\frac{\left[S^{-\frac{q}{p}}|\Omega|^{\frac{\alpha+\beta-q}{\alpha+\beta}}\left(\lambda^{\frac{p}{p-q}}+\mu^{\frac{p}{p-q}}\right)^{\frac{p-q}{p}}\|(u,v)\|^q\right]^{\frac{p-(\alpha+\beta)}{p-q}}}{\|(u,v)\|^{\frac{p(q-(\alpha+\beta))}{p-q}}}\nonumber\\
		&=  \frac{p-q}{\alpha+\beta-q}\left(\frac{\alpha+\beta-q}{ \alpha+\beta-p }\right)^{\frac{p-(\alpha+\beta)}{p-q}}\left[S^{-\frac{q}{p}}|\Omega|^{\frac{\alpha+\beta-q}{\alpha+\beta}}\right]^{\frac{p-(\alpha+\beta)}{p-q}}\left(\lambda^{\frac{p}{p-q}}+\mu^{\frac{p}{p-q}}\right)^{\frac{p-(\alpha+\beta)}{p}}\|(u,v)\|^{\alpha+\beta}.
		\end{align}
		On the other hand, using Young inequality and the definition of $S$, it holds that 
		\begin{equation*}
		2\int_\Omega  |u|^{\alpha} |v|^\beta dx \leq 2\left(\frac{\alpha}{\alpha+\beta}
		\int_{\Omega}|u|^{\alpha+\beta}dx+\frac{\beta}{\alpha+\beta}
		\int_{\Omega}|v|^{\alpha+\beta}dx\right)\leq 2 S^{-\frac{\alpha+\beta}{p}}\|(u,v)\|^{\alpha+\beta}.
		\end{equation*}
		For any  $\lambda,\mu$ satisfying 
		$0<\lambda^{\frac{p}{p-q}}+\mu^{\frac{p}{p-q}}<\Lambda_1$
		with $\Lambda_1$ given in \eqref{lambda1}, we have
		\begin{align}\label{nehar40abc0}
		2 S^{-\frac{\alpha+\beta}{p}} 
		&\leq\frac{p-q}{\alpha+\beta-q}\left(\frac{\alpha+\beta-q}{ \alpha+\beta-p }\right)^{\frac{p-(\alpha+\beta)}{p-q}}\left[S^{-\frac{q}{p}}|\Omega|^{\frac{\alpha+\beta-q}{\alpha+\beta}}\right]^{\frac{p-(\alpha+\beta)}{p-q}}\nonumber\\&\quad \times \left(\lambda^{\frac{p}{p-q}}+\mu^{\frac{p}{p-q}}\right)^{\frac{p-(\alpha+\beta)}{p}}.
		\end{align}
		Thus, from (\ref{nehar40a})-(\ref{nehar40abc0}), if $\lambda,\mu$ satisfy 
		$0<\lambda^{\frac{p}{p-q}}+\mu^{\frac{p}{p-q}}<\Lambda_1$, we have
		\begin{align*}
		0<2\int_\Omega  |u|^{\alpha} |v|^\beta dx &\leq   2 S^{-\frac{\alpha+\beta}{p}}\|(u,v)\|^{\alpha+\beta}
		\nonumber\\
		&\leq\frac{p-q}{\alpha+\beta-q}\left(\frac{\alpha+\beta-q}{ \alpha+\beta-p }\right)^{\frac{p-(\alpha+\beta)}{p-q}}\left[S^{-\frac{q}{p}}|\Omega|^{\frac{\alpha+\beta-q}{\alpha+\beta}}\right]^{\frac{p-(\alpha+\beta)}{p-q}} \\&\quad \times \left(\lambda^{\frac{p}{p-q}}+\mu^{\frac{p}{p-q}}\right)^{\frac{p-(\alpha+\beta)}{p}}\|(u,v)\|^{\alpha+\beta}  \\
		&<\Psi_{u,v}(t_{\max}(u,v)).
		\end{align*}
		Then, there exist unique $t_1>0$ and $t_2>0$ with $t_1<t_{\max}(u,v)<t_2$, such that
		$$
		\Psi_{u,v}(t_1)=\Psi_{u,v}(t_2)=2\int_\Omega  |u|^{\alpha} |v|^\beta dx,\quad  \Psi_{u,v}'(t_1)>0,\ \ \ \Psi_{u,v}'(t_2)<0.
		$$
      In turn, \eqref{nehar2} and \eqref{nehar40} give that
		$
		\varphi_{u,v}'(t_1)=\varphi_{u,v}'(t_2)=0.
		$
		By (\ref{nehar4ae}) we have that
		$\varphi_{u,v}''(t_1)>0$ and $\varphi_{u,v}''(t_2)<0$.
		These facts imply that $\varphi_{u,v}$ 
		has a local minimum at $t_1$ and a local maximum at $t_2$ such that $ (t_1 u,t_1v)\in\mathcal{N}_{\lambda,\mu}^+$ and $ (t_2 u,t_2v)\in\mathcal{N}_{\lambda,\mu}^-$,. Since $\varphi_{u,v}(t)= J_{\lambda,\mu}(tu,tv)$, we have $J_{\lambda,\mu}(t_2u,t_2v)\geq J_{\lambda,\mu}(tu,tv)\geq J_{\lambda,\mu}(t_1u,t_1v)$ for each $t\in[t_1,t_2]$ and
		$J_{\lambda,\mu}(t_1u,t_1v)\leq J_{\lambda,\mu}(t u,t v)$ for each $t\in[0,t_1]$. Thus
		$$
		J_{\lambda,\mu}(t_1u, t_1v)=\inf\limits_{0\leq t\leq t_{\max}}J_{\lambda,\mu}(tu,tv),\qquad
		J_{\lambda,\mu}(t_2u, t_2v)=\sup\limits_{ t\geq 0}J_{\lambda,\mu}(tu,tv).
		$$
		The graphs of $\Psi_{u,v}$ and $\varphi_{u,v}$ can be seen in Figure 1.
\end{proof}

%
%
%
%
%


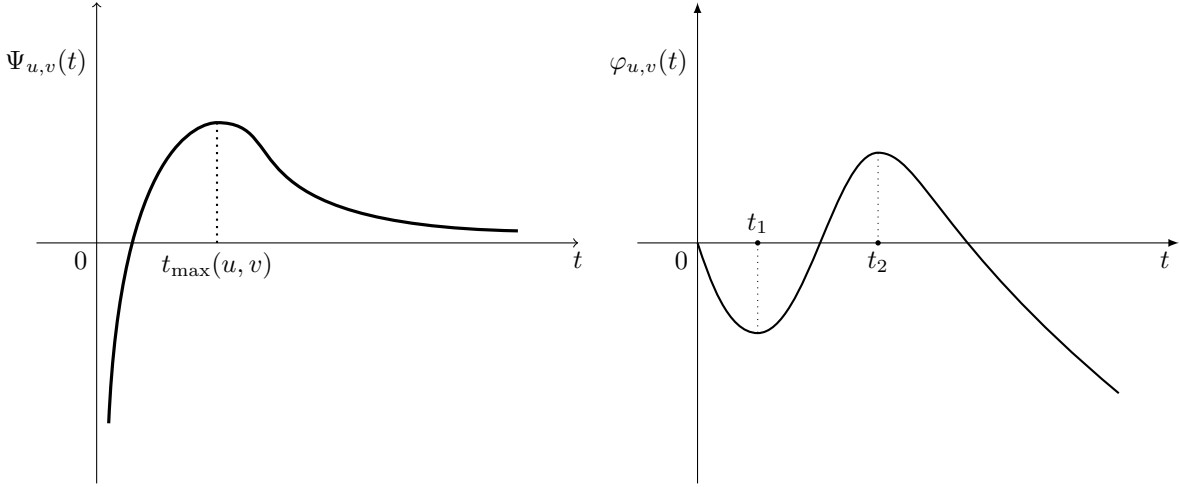
\begin{figure}
	\label{figura1}
	\centering
	\begin{tikzpicture}[scale=0.8]
	\draw[->] (-1, 0) -- (8, 0);
	\draw[->] (0, -4) -- (0, 4);
	\draw[very thick] (0.2 ,-3) .. controls (0.4,1.6) and (1.7,2) .. (2, 2) .. controls (3.3, 2) and (2, 0.3) .. (7, 0.2); 
	\draw[thick, dotted] (2,0) -- (2,2);
	
	\draw (8,0) node[below]{$t$};
	\draw (0, 0) node[below left]{$0$};
	\draw (0,3) node[left]{$\Psi_{u, v}(t)$};
	\draw (2,0) node[below]{$t_{\rm max}(u,v)$};
	\end{tikzpicture}
	\begin{tikzpicture}[scale=0.8]
	\draw[-latex] (-1, 0) -- (8, 0);
	\draw[-latex] (0, -4) -- (0, 4);
	\draw[thick] (0 ,0) .. controls (0.2,-0.6) and (0.5,-1.5) .. (1, -1.5) .. controls (1.8, -1.5) and (2.3, 1.5) .. (3, 1.5) .. controls (3.7, 1.5) and (4, 0)  .. (7, -2.5); 
	\draw[dotted] (1,0) -- (1,-1.5);
	\draw[dotted] (3,0) -- (3, 1.5);
	\filldraw (3,0) circle (1pt);
	\filldraw (1,0) circle (1pt);
	\draw (8,0) node[below left]{$t$};
	\draw (0, 0) node[below left]{$0$};
	\draw (0,3) node[left]{$\varphi_{u, v}(t)$};
	\draw (1,0) node[above]{$t_1$};
	\draw (3,0) node[below]{$t_2$};
	\end{tikzpicture}
	\caption{The graphs of $\Psi_{u,v}$ and $\varphi_{u,v}$}
\end{figure}

\section{The Palais-Smale condition}\label{ps}

In this section, we show that the functional $J_{\lambda,\mu}$ satisfies $(PS)_c$ condition.

\begin{definition}\label{defps}
Let $c\in\mathbb{R}$ , $E$ be a Banach space and $J_{\lambda,\mu}\in  C^1(E,\mathbb{R})$.
$\{(u_k,v_k)\}_{k\in {\mathbb N}}$ is a $(PS)_c$ sequence in $E$ for $J_{\lambda,\mu}$ 
if $J_{\lambda,\mu}(u_k,v_k)=c+o(1)$ and $J_{\lambda,\mu}'(u_k,v_k)=o(1)$ strongly in $E^\ast$ as $k\to\infty$.
We say that $J_{\lambda,\mu}$ satisfies the $(PS)_c$ condition if any $(PS)_c$ sequence 
$\{(u_k,v_k)\}_{k\in {\mathbb N}}$ for $J_{\lambda,\mu}$ in $E$ admits a convergent subsequence.
\end{definition}

\begin{lemma}[Boundedness of $(PS)_c$ sequences]
	\label{ps3}
If $\{(u_k,v_k)\}_{k\in {\mathbb N}}\subset E$ is a $(PS)_c$ sequence for $J_{\lambda,\mu}$, then $\{(u_k,v_k)\}_{k\in {\mathbb N}}$ is bounded in $E$.
\end{lemma}
\begin{proof}
If $\{(u_k,v_k)\}\subset E$ is a $(PS)_c$ sequence for $J_{\lambda,\mu}$, then we have
\begin{eqnarray*}
J_{\lambda,\mu}(u_k,v_k)\to c,\quad J'_{\lambda,\mu}(u_k,v_k)\to 0\ \ \ \mbox{in}\ \ E^\ast\quad \mbox{as}\ k\to\infty.
\end{eqnarray*}
That is,
\begin{align}\label{pscr4}
 \frac{1}{p}\|(u_k,v_k)\|^p-\frac{1}{q}\int_\Omega(\lambda|u_k|^q+\mu|v_k|^q)dx-\frac{2}{\alpha+\beta}\int_\Omega|u_k|^\alpha|v_k|^\beta dx&= c+o_k(1), \\
\label{pscr4ad}
 \|(u_k,v_k)\|^p- \int_\Omega(\lambda|u_k|^q+\mu|v_k|^q)dx-2\int_\Omega|u_k|^\alpha|v_k|^\beta dx&= o_k(\|(u_k,v_k)\|),
\end{align}
as $k\to\infty$.
We show that $(u_k,v_k)$ is bounded in $E$
by contradiction. Assume $\|(u_k,v_k)\|\to\infty$, set 
$$
\tilde{u}_k:=\frac{u_k}{\|(u_k,v_k)\|},  \,\,\quad
 \tilde{v}_k:=\frac{v_k}{\|(u_k,v_k)\|},
$$ 
then $\|(\tilde{u}_k,\tilde{v}_k)\|=1$. There is a subsequence, still denote by itself, 
with $(\tilde{u}_k,\tilde{v}_k)\rightharpoonup (\tilde{u},\tilde{v})\in E$ and
\begin{eqnarray*}
\tilde{u}_k\to \tilde{u},\ \ \tilde{v}_k\to \tilde{v}\quad \mbox{in} \ \ L^r(\mathbb{R}^n),\qquad
\tilde{u}_k\to \tilde{u},\ \ \tilde{v}_k\to \tilde{v}\ \ \mbox{a.e.\ in}\ \ \mathbb{R}^n,
\end{eqnarray*}
for any $1\leq  r<p_s^\ast=\frac{np}{n-ps}$.
Then, the Dominated Convergence Theorem yields
\begin{eqnarray}\label{pscr4a}
\int_\Omega(\lambda|\tilde{u}_k|^{q}+\mu|\tilde{v}_k|^q)dx\to \int_\Omega(\lambda|\tilde{u}|^{q}+\mu|\tilde{v}|^q)dx,\quad \mbox{as}\ \ k\to\infty.
\end{eqnarray}
Moreover, from \eqref{pscr4} and \eqref{pscr4ad}, we find that $({\tilde u_k},{\tilde v_k})$ satisfy
\begin{align*}
&\frac{1}{p}\|(\tilde{u}_k,\tilde{v}_k)\|^p-\frac{\|(u_k,v_k)\|^{q-p}}{q}
\int_\Omega(\lambda|\tilde{u}_k|^q+\mu|\tilde{v}_k|^q)dx-\frac{2\|(u_k,v_k)\|^{\alpha+\beta-p}}{\alpha+\beta}\int_\Omega|\tilde{u}_k|^\alpha|\tilde{v}_k|^\beta dx=o_k(1), \\
&\|(\tilde{u}_k,\tilde{v}_k)\|^p- \|(u_k,v_k)\|^{q-p}\int_\Omega(\lambda|\tilde{u}_k|^q+\mu|\tilde{v}_k|^q)dx-
2\|(u_k,v_k)\|^{\alpha+\beta-p}\int_\Omega|\tilde{u}_k|^\alpha|\tilde{v}_k|^\beta dx= o_k(1).
\end{align*}
From above two equalities and (\ref{pscr4a}), we obtain
\begin{align*}
\|(\tilde{u}_k,\tilde{v}_k)\|^p
&=\frac{p(\alpha+\beta-q)}{q(\alpha+\beta-p)}\|(u_k,v_k)\|^{q-p}\int_\Omega(\lambda|\tilde{u}_k|^q+\mu|\tilde{v}_k|^q)dx + o_k(1)\\
&= \frac{p(\alpha+\beta-q)}{q(\alpha+\beta-p)}\|(u_k,v_k)\|^{q-p}\int_\Omega(\lambda|\tilde{u}|^q+\mu|\tilde{v}|^q)dx  + o_k(1).
\end{align*}
Since $1<q<p$ and $\|(u_k,v_k)\|\to\infty$, then we get $\|(\tilde{u}_k,\tilde{v}_k)\|^p\to 0$,
which contradicts $\|(\tilde{u}_k,\tilde{v}_k)\|=1$.
\end{proof}

\begin{lemma}[Uniform lower bound]
	\label{ps2}
If $\{(u_k,v_k)\}_{k\in {\mathbb N}}$ is a $(PS)_c$ sequence for $J_{\lambda,\mu}$ with $(u_k,v_k)\rightharpoonup (u,v)$ in $E$, then $J'_{\lambda,\mu}(u,v)=0$, and there exists a positive constant $C_0$ depending on $p,q,s,n,S$ and $|\Omega|$ such that
\begin{eqnarray}\label{pseq1}
J_{\lambda,\mu}(u,v)\geq -C_0\left(\lambda^{\frac{p}{p-q}}+\mu^{\frac{p}{p-q}}\right), 
\end{eqnarray}
where we have set 
\begin{equation}
\label{cizero}
C_0:=
\frac{p-q}{pq p^*_s}\frac{(p_s^\ast-q)^{\frac{p}{p-q}}}{(p_s^\ast-p)^{\frac{q}{p-q}}}
|\Omega|^{\frac{p(p_s^\ast-q)}{ p_s^\ast (p-q)}}S^{-\frac{q}{p-q}},
\end{equation}
being $S$ the best constant for the Sobolev embedding of $X_0$ into $L^{p^*_s}({\mathbb R}^n)$.
\end{lemma}

\begin{proof}
If $\{(u_k,v_k)\}\subset E$ is a $(PS)_c$ sequence for $J_{\lambda,\mu}$ with $(u_k,v_k)\rightharpoonup (u,v)$ in $E$. That is
\begin{equation*}
J_{\lambda,\mu}'(u_k,v_k)=o(1)\quad \mbox{strongly \ in}\  E^\ast\ \ \mbox{as}\  k\to\infty.
\end{equation*}
Let $(\phi,\psi)\in E$, then it holds that
\begin{align*}
&\langle J_{\lambda,\mu}'(u_k,v_k)-J_{\lambda,\mu}'(u,v),(\phi,\psi)\rangle\nonumber\\
&=\mathcal{A}(u_k,\phi) -\mathcal{A}(u,\phi)+\mathcal{A}(v_k,\psi)-\mathcal{A}(v,\psi)\nonumber\\
&-\lambda \int_\Omega \left( |u_k|^{q-2}u_k-|u|^{q-2}u\right)\phi dx-\mu \int_\Omega \left( |v_k|^{q-2}v_k-|v|^{q-2}v\right)\psi dx\nonumber\\
&-\frac{2\alpha}{\alpha+\beta}\int_\Omega\left(|u_k|^{\alpha-2}u_k|v_k|^\beta-|u|^{\alpha-2}u|v|^\beta\right)\phi dx
-\frac{2\beta}{\alpha+\beta}\int_\Omega\left(|u_k|^{\alpha}|v_k|^{\beta-2}v_k-|u|^{\alpha}|v|^{\beta-2}v\right)\psi dx,
\end{align*}
where $\mathcal{A}$ is defined in (\ref{huaa}).
We claim that, from $(u_k,v_k)\rightharpoonup (u,v)$ in $E$, we  have 
$$
\lim_k\mathcal{A}(u_k,\phi)=\mathcal{A}(u,\phi),\qquad \lim_k\mathcal{A}(v_k,\psi)=\mathcal{A}(v,\psi),
$$
for any $\phi,\psi\in X_0$ as $k\to\infty$. In fact,  the sequences
$$
\left\{\frac{|u_k(x)-u_k(y)|^{p-2}(u_k(x)-u_k(y))}{|x-y|^{\frac{n+ps}{p'}}}\right\}_{k\in {\mathbb N}} \qquad
\left\{\frac{|v_k(x)-v_k(y)|^{p-2} (v_k(x)-v_k(y))}{|x-y|^{\frac{n+ps}{p'}}}\right\}_{k\in {\mathbb N}}
$$
are bounded in $L^{p'}({\mathbb R}^n)$ and by the poinwise converge $u_k\to u$ and $v_k\to v$, there holds
$$
\frac{|u_k(x)-u_k(y)|^{p-2}(u_k(x)-u_k(y))}{|x-y|^{\frac{n+ps}{p'}}}\overset{L^{p'}({\mathbb R}^n)}{\rightharpoonup}
\frac{|u(x)-u(y)|^{p-2}(u(x)-u(y))}{|x-y|^{\frac{n+ps}{p'}}},
$$
and
$$
\frac{|v_k(x)-v_k(y)|^{p-2}(v_k(x)-v_k(y))}{|x-y|^{\frac{n+ps}{p'}}}\overset{L^{p'}({\mathbb R}^n)}{\rightharpoonup}
\frac{|v(x)-v(y)|^{p-2}(v(x)-v(y))}{|x-y|^{\frac{n+ps}{p'}}}.
$$
Since
$$
\frac{\phi(x)-\phi(y)}{|x-y|^{\frac{n+ps}{p}}}\in L^{p}({\mathbb R}^n),\qquad
\frac{\psi(x)-\psi(y)}{|x-y|^{\frac{n+ps}{p}}}\in L^{p}({\mathbb R}^n),
$$
the claim follows. The sequences $u_k$ and $v_k$ are bounded in $X_0$, and then in $L^{p_s^\ast}(\Omega)$.
Then $u_k\to u$ and $v_k\to v$ weakly in $L^{p_s^\ast}(\mathbb{R}^n)$.
Furthermore,  we obtain
$$
|u_k|^{q-2}u_k\overset{L^{q'}(\Omega)}{\rightharpoonup} |u|^{q-2}u,\ \qquad 
|v_k|^{q-2}v_k\overset{L^{q'}(\Omega)}{\rightharpoonup} |v|^{q-2}v,
$$
$$
|u_k|^{\alpha-2}u_k|v_k|^\beta\overset{L^{\frac{\alpha+\beta}{\alpha+\beta-1}}(\Omega)}{\rightharpoonup} |u|^{\alpha-2}u|v|^\beta, \qquad
|u_k|^{\alpha}|v_k|^{\beta-2}v_k\overset{L^{\frac{\alpha+\beta}{\alpha+\beta-1}}(\Omega)}{\rightharpoonup} |u|^{\alpha}|v|^{\beta-2}v,
$$
Since $\phi,\psi\in X_0\subset L^{q}(\Omega)\cap L^{\alpha+\beta}(\Omega)$, it follows that, as $k\to\infty$,
$$
\int_{\Omega}\left( |u_k|^{q-2}u_k-|u|^{q-2}u\right)\phi dx\to0,\qquad
\int_{\Omega}\left( |v_k|^{q-2}v_k-|v|^{q-2}v\right)\psi dx\to0,
$$
and
$$
\int_\Omega\left(|u_k|^{\alpha-2}u_k|v_k|^\beta-|u|^{\alpha-2}u|v|^\beta\right)\phi dx\to0,
\qquad
\int_\Omega\left(|u_k|^{\alpha}|v_k|^{\beta-2}v_k-|u|^{\alpha}|v|^{\beta-2}v\right)\psi dx\to0.
$$
Hence $\langle J_{\lambda,\mu}'(u_k,v_k)-J_{\lambda,\mu}'(u,v),(\phi,\psi)\rangle\to0$ for all $(\phi,\psi)\in E$,
which yields  $J'_{\lambda,\mu}(u,v)=0$. In particular, we get $\langle J'_{\lambda,\mu}(u,v), (u,v)\rangle=0$, namely
$$
2\int_\Omega|u|^\alpha|v|^\beta dx=\|(u,v)\|^p-\int_\Omega(\lambda|u|^q+\mu|v|^q)dx.
$$
Then
\begin{align}\label{pseq2}
J_{\lambda,\mu}(u,v)&=\left(\frac{1}{p}-\frac{1}{p_s^\ast}\right)\|(u,v)\|^p-\left(\frac{1}{q}-\frac{1}{p_s^\ast}\right)\int_\Omega(\lambda|u|^q+\mu|v|^q)dx\nonumber\\
&=\frac{s}{n}\|(u,v)\|^p-\left(\frac{1}{q}-\frac{1}{p_s^\ast}\right)\int_\Omega(\lambda|u|^q+\mu|v|^q)dx.
\end{align}
By H\"{o}lder inequality, Sobolev embedding,  (\ref{criticalfrac}) and Young inequality, we have
\begin{align}\label{pseq3}
&\int_\Omega(\lambda|u|^q+\mu|v|^q)dx
\leq  |\Omega|^{\frac{p_s^\ast-q}{p_s^\ast}}S^{-\frac{q}{p}}\left(\lambda\|u\|_{X_0}^q+
\mu\|v\|^q_{X_0}\right)\nonumber\\
&=\left(\left[\frac{p}{q} \frac{s}{n}\left(\frac{1}{q}-\frac{1}{p_s^\ast}\right)^{-1} \right]^{\frac{q}{p}}\|u\|_{X_0}^q\right)\left(\left[\frac{p}{q} \frac{s}{n}\left(\frac{1}{q}-\frac{1}{p_s^\ast}\right)^{-1}\right]^{-\frac{q}{p}}
|\Omega|^{\frac{p_s^\ast-q}{p_s^\ast}}S^{-\frac{q}{p}}\lambda\right)\nonumber\\
&+\left(\left[\frac{p}{q} \frac{s}{n}\left(\frac{1}{q}-\frac{1}{p_s^\ast}\right)^{-1}\right]^{\frac{q}{p}}\|v\|_{X_0}^q\right)\left(\left[\frac{p}{q} \frac{s}{n}\left(\frac{1}{q}-\frac{1}{p_s^\ast}\right)^{-1}\right]^{-\frac{q}{p}}
|\Omega|^{\frac{p_s^\ast-q}{p_s^\ast}}S^{-\frac{q}{p}}\mu\right)\nonumber\\
&\leq \frac{s}{n}\left(\frac{1}{q}-\frac{1}{p_s^\ast}\right)^{-1} \left(\|u\|_{X_0}^p+\|v\|_{X_0}^p\right) +\widehat{C}\big(\lambda^{\frac{p}{p-q}}+\mu^{\frac{p}{p-q}}\big)\nonumber\\
&=\frac{s}{n}\left(\frac{1}{q}-\frac{1}{p_s^\ast}\right)^{-1}  \|(u,v)\|^p +\widehat{C}\big(\lambda^{\frac{p}{p-q}}+\mu^{\frac{p}{p-q}}\big),
\end{align}
with
\begin{equation*}
\widehat{C}= \frac{p-q}{p}\left(\left[\frac{p}{q} \frac{s}{n}\left(\frac{1}{q}-\frac{1}{p_s^\ast}\right)^{-1}\right]^{-\frac{q}{p}}
|\Omega|^{\frac{p_s^\ast-q}{p_s^\ast}}S^{-\frac{q}{p}} \right)^{\frac{p}{p-q}}  
 =  \frac{p-q}{p}\left(\frac{p_s^\ast-q}{p_s^\ast-p}\right)^{\frac{q}{p-q}}
|\Omega|^{\frac{p(p_s^\ast-q)}{ p_s^\ast (p-q)}}S^{-\frac{q}{p-q}}.
\end{equation*}
Then \eqref{pseq1} follows from \eqref{pseq2} and \eqref{pseq3} with $C_0=\big(\frac{1}{q}-\frac{1}{p_s^\ast}\big)\widehat{C}$.
\end{proof}

\noindent
Let us set
\begin{equation}\label{criticalfracertga}
S_{\alpha,\beta} :=\inf\limits_{(u,v)\in E\backslash\{0\}} \frac{\|(u,v)\|^p}
{\left(\displaystyle\int_{\Omega}|u|^{\alpha}|v|^\beta dx\right)^{\frac{p}{\alpha+\beta}}}.
\end{equation}
\noindent
We have the following result which provides a connection between $S_{\alpha,\beta}$ and $S$. The 
argument follows essentially the line of \cite{ahn} but, for the sake of self-containedness, we include it.

\begin{lemma}[$S_{\alpha,\beta}$ versus $S$]
There holds
\begin{equation}
\label{alphabeta}
S_{\alpha,\beta}=\Big[\Big(\frac{\alpha}{\beta}\Big)^{\frac{\beta}{\alpha+\beta}}+\Big(\frac{\beta}{\alpha}\Big)^{\frac{\alpha}{\alpha+\beta}}\big]S.
\end{equation}
\end{lemma}
\begin{proof}
		Let $\{\omega_n\}_{n\in {\mathbb N}}\subset X_0$  be a
		minimization sequence for $S.$ Let $s,t>0$ be chosen later and consider
		the sequences $u_{n}:=s  \omega_n$ and $v_{n}:=t \omega_n$ in $X_0$. By the definition of $S_{\alpha,\beta}$, we have
		\begin{equation}\label{for1}
		\frac{s^p+t^p}{(s^{\alpha}t^{\beta})^{\frac{p}{p^*_s}}}\frac{\displaystyle\int_{\mathbb{R}^{2n}}\frac{|\omega_n(x)-\omega_n(y)|^p}{|x-y|^{n+ps}}dxdy}
		{\left(\displaystyle\int_{\Omega}|\omega_n|^{p^*_s}dx\right)^{\frac{p}{p^*_s}}}\geq
		S_{\alpha,\beta}.
		\end{equation}
		Observe that
		$$
		\frac{s^p+t^p}{(s^{\alpha}t^{\beta})^{\frac{p}{p^*_s}}}=\left(\frac{s}{t}\right)^{\frac{p\beta}{p_s^\ast}}
		+\left(\frac{s}{t}\right)^{-\frac{p\alpha}{p_s^\ast}}.
		$$
		Let us consider the function
		$g:\mathbb{R}^+\to\mathbb{R}^+$ by setting $g(x):=x^{\frac{p\beta}{p^*_s}}+x^{\frac{-p\alpha}{p^*_s}}$, we have
		$$
		\frac{s^p+t^p}{(s^{\alpha}t^{\beta})^{\frac{p}{p^*_s}}}=g\Big(\frac{s}{t}\Big),
		$$
		and   the function $g$ achieves its minimum at  point $x_0=\big(\frac{\alpha}{\beta}\big)^{\frac{1}{p}}$ with minimum value 
		$$
		\min_{x\in\mathbb{R}^+}g(x)=
		\left(\frac{\alpha}{\beta}\right)^{\frac{\beta}{p^*_s}}+\left(\frac{\beta}{\alpha}\right)^{\frac{\alpha}{p^*_s}}.
		$$
		Choosing $s,t$ in \eqref{for1} such that $\frac{s}{t}=\big(\frac{\alpha}{\beta}\big)^{\frac{1}{p}}$ and letting $n\to\infty$ yields
		\begin{equation}\label{for2}
		\Big[\left(\frac{\alpha}{\beta}\right)^{\frac{\beta}{p^*_s}}+
		\left(\frac{\beta}{\alpha}\right)^{\frac{\alpha}{p^*_s}}\Big]S\geq
		S_{\alpha,\beta}.
		\end{equation}
		On the other hand, let $\{(u_{n},v_{n})\}_{n\in {\mathbb N}}\subset E\setminus\{(0,0)\}$ be
		a minimizing sequence for $S_{\alpha,\beta}.$ Set
		$z_n:=s_nv_{n}$ for $s_n>0$ with $\int_{\Omega}|u_{n}|^{p^*_s}dx=\int_{\Omega}|z_n|^{p^*_s}dx.$
		Then Young inequality implies
		\begin{equation*}
		\int_{\Omega}|u_{n}|^{\alpha}|z_n|^{\beta}dx\leq \frac{\alpha}{\alpha+\beta}
		\int_{\Omega}|u_{n}|^{\alpha+\beta}dx+\frac{\beta}{\alpha+\beta}
		\int_{\Omega}|z_n|^{\alpha+\beta}dx= \int_{\Omega}|z_n|^{\alpha+\beta}dx=\int_{\Omega}|u_{n}|^{\alpha+\beta}dx.
		\end{equation*}
		Then we have
		\begin{align*}
			&\frac{\displaystyle\int_{\mathbb{R}^{2n}}\frac{|u_n(x)-u_n(y)|^p}{|x-y|^{n+ps}}dxdy+\int_{\mathbb{R}^{2n}}\frac{|v_n(x)-v_n(y)|^p}{|x-y|^{n+ps}}dxdy}
			{\Big(\displaystyle\int_{\Omega}|u_{n}|^{\alpha}|v_{n}|^{\beta}dx\Big)^{\frac{p}{\alpha+\beta}}}\\
			&=\frac{s_n^{\frac{p\beta}{\alpha+\beta}}\left(\displaystyle\int_{\mathbb{R}^{2n}}\frac{|u_n(x)-u_n(y)|^p}{|x-y|^{n+ps}}dxdy+\int_{\mathbb{R}^{2n}}\frac{|v_n(x)-v_n(y)|^p}{|x-y|^{n+ps}}dxdy\right)}
			{\Big(\displaystyle\int_{\Omega}|u_{n}|^{\alpha}|z_{n}|^{\beta}dx\Big)^{\frac{p}{\alpha+\beta}}}\\
			&\geq s_n^{\frac{p\beta}{\alpha+\beta}}\frac{\displaystyle\int_{\mathbb{R}^{2n}}\frac{|u_n(x)-u_n(y)|^p}{|x-y|^{n+ps}}dxdy}
			{\Big(\displaystyle\int_{\Omega}|u_{n}|^{\alpha+\beta}dx\Big)^{\frac{p}{\alpha+\beta}}}
			+s_n^{\frac{p\beta}{\alpha+\beta}}s_n^{-p}\frac{\displaystyle\int_{\mathbb{R}^{2n}}\frac{|z_n(x)-z_n(y)|^p}{|x-y|^{n+ps}}dxdy}
			{\Big(\displaystyle\int_{\Omega}|z_{n}|^{\alpha+\beta}dx\Big)^{\frac{p}{\alpha+\beta}}}\\
			&\geq g(s_n)S\geq \Big[\left(\frac{\alpha}{\beta}\right)^{\frac{\beta}{p^*_s}}+\left(\frac{\beta}{\alpha}\right)^{\frac{\alpha}{p^*_s}}\Big]S.
		\end{align*}
		Passing to the limit as $n\to\infty$ in the last inequality we obtain
		\begin{equation}\label{for3}
		\Big[\left(\frac{\alpha}{\beta}\right)^{\frac{\beta}{p^*_s}}+
		\left(\frac{\beta}{\alpha}\right)^{\frac{\alpha}{p^*_s}}\Big]S\leq
		S_{\alpha,\beta}.
		\end{equation}
		Thus \eqref{alphabeta} follows from \eqref{for2} and \eqref{for3}.
\end{proof}

\begin{lemma}[Palais-Smale range]
	\label{pscond}
$J_{\lambda,\mu}$ satisfies the $(PS)_c$ condition with $c$ satisfying
\begin{eqnarray}\label{pscv}
-\infty<c<c_\infty= \frac{2s}{n}\left(\frac{S_{\alpha,\beta}}{2}\right)^{\frac{n}{ps}}
-C_0\left(\lambda^{\frac{p}{p-q}}+\mu^{\frac{p}{p-q}}\right),
\end{eqnarray}
where $C_0$ is the positive constant defined in \eqref{cizero}.
\end{lemma}

\begin{proof}
Let $\{(u_k,v_k)\}_{k\in {\mathbb N}}$ be a $(PS)_c$ sequence of $J_{\lambda,\mu}$ in $E$. Then
\begin{align}\label{pscr4r}
\frac{1}{p}\|(u_k,v_k)\|^p-\frac{1}{q}\int_\Omega(\lambda|u_k|^q+\mu|v_k|^q)dx-\frac{2}{p_s^\ast}\int_\Omega|u_k|^\alpha|v_k|^\beta dx&= c+o_k(1), \\
\label{pscr4ar}
\|(u_k,v_k)\|^p- \int_\Omega(\lambda|u_k|^q+\mu|v_k|^q)dx-2\int_\Omega|u_k|^\alpha|v_k|^\beta dx&= o_k(1).
\end{align}
We know by Lemma \ref{ps3}  that $\{(u_k,v_k)\}_{k\in {\mathbb N}}$ is bounded in $E$. 
Then, up to a subsequence, $(u_k,v_k)\rightharpoonup (u,v)$ in $E$, and by Lemma \ref{ps2} we learn that $(u,v)$ is a critical point of $J_{\lambda,\mu}$. 
Next we show that $(u_k,v_k)$ converges strongly to $(u,v)$ as $k\to\infty$ in $E$. Since $u_k\to u$ and $v_k\to v$ in $L^r(\mathbb{R}^n)$, we obtain 
$$
\int_\Omega(\lambda|u_k|^{q}+\mu|v_k|^q)dx\to \int_\Omega(\lambda|u|^{q}+\mu|v|^q)dx,\quad \mbox{as}\ \ k\to\infty.
$$
Moreover, by variants of the Brezis-Lieb Lemma, we can easily get
\begin{eqnarray}\label{converafg}
\|(u_k,v_k)\|^p=\|(u_k-u,v_k-v)\|^p+\|(u,v)\|^p+o_k(1),
\end{eqnarray}
(cf.\ \cite[Lemma 2.2]{bsy}), and
\begin{eqnarray}\label{converafga}
\int_{\Omega}|u_k|^\alpha|v_k|^\beta dx=\int_{\Omega}|u_k-u|^\alpha|v_k-v|^\beta dx+\int_{\Omega}|u|^\alpha|v|^\beta dx+o_k(1).
\end{eqnarray}
Taking (\ref{converafg}) and (\ref{converafga}) into (\ref{pscr4r}) and (\ref{pscr4ar}), we find
\begin{eqnarray}\label{converafgb}
\frac{1}{p}\|(u_k-u,v_k-v)\|^p-\frac{2}{p_s^\ast}\int_{\Omega}|u_k-u|^\alpha|v_k-v|^\beta dx=c-J_{\lambda,\mu}(u,v)+o_k(1),
\end{eqnarray}
and
\begin{equation*}
\|(u_k-u,v_k-v)\|^p=2\int_{\Omega}|u_k-u|^\alpha|v_k-v|^\beta dx+o_k(1).
\end{equation*}
Hence, we may assume that
\begin{eqnarray}\label{converafgablg}
\|(u_k-u,v_k-v)\|^p\to m,\quad 2\int_{\Omega}|u_k-u|^\alpha|v_k-v|^\beta dx\to m\quad \mbox{as}\ k\to\infty.
\end{eqnarray}
If $m=0$, we are done. Suppose $m>0$. Then from (\ref{converafgablg}) and the definition of $S_{\alpha,\beta}$ in (\ref{criticalfracertga}), we have
\begin{eqnarray*}
S_{\alpha,\beta}\left(\frac{m}{2}\right)^{\frac{p}{p_s^\ast}}= S_{\alpha,\beta}\lim\limits_{k\to\infty}\Big(\int_{\Omega}|u_k-u|^\alpha|v_k-v|^\beta dx\Big)^{\frac{p}{p_s^\ast}}\leq \lim\limits_{k\to\infty}
\|(u_k-u,v_k-v)\|^{p}=m,
\end{eqnarray*}
this yields that $m\geq 2 \big(\frac{S_{\alpha,\beta}}{2}\big)^{\frac{n}{ps}}$.
From \eqref{converafgb}, we obtain
\begin{equation*}
c= \frac{s}{n} m+J_{\lambda,\mu}(u,v).
\end{equation*}
By Lemma \ref{ps2} and $m\geq 2 \big(\frac{S_{\alpha,\beta}}{2}\big)^{\frac{n}{ps}}$, we find
\begin{equation*}
c \geq \frac{2s}{n}\Big(\frac{S_{\alpha,\beta}}{2}\Big)^{\frac{n}{ps}}-C_0\big(\lambda^{\frac{p}{p-q}}+\mu^{\frac{p}{p-q}}\big),
\end{equation*}
which is impossible for $-\infty<c< \frac{2s}{n}\big(\frac{S_{\alpha,\beta}}{2}\big)^{\frac{n}{ps}}
-C_0\big(\lambda^{\frac{p}{p-q}}+\mu^{\frac{p}{p-q}}\big)$.
\end{proof}

\smallskip

\section{Existence of solutions}
\label{Exx}

\noindent
Next we start with some Lemmas.

\begin{lemma}[$\mathcal{N}_{\lambda,\mu}^0$ is empty]
	\label{le1}
Let $0<\lambda^{\frac{p}{p-q}}+\mu^{\frac{p}{p-q}}<\Lambda_1$, where $\Lambda_1$ is as in \eqref{lambda1}. Then  $\mathcal{N}_{\lambda,\mu}^0=\emptyset$.
\end{lemma}
\begin{proof}
We learn from the proof of Lemma~\ref{firb2} that there exist {\em exactly} two numbers
$t_2>t_1>0$ such that $\varphi_{u,v}'(t_1)=\varphi_{u,v}'(t_2)=0$. Furthermore, $\varphi_{u,v}''(t_1)>0>\varphi_{u,v}''(t_2)$.
If by contradiction $(u,v)\in\mathcal{N}_{\lambda,\mu}^0$ we have $\varphi'_{u,v}(1)=0$ with $\varphi_{u,v}''(1)=0$. Then, either $t_1=1$ or $t_2=1$. 
In turn, either $\varphi_{u,v}''(1)>0$ or $\varphi_{u,v}''(1)<0$, a contradiction.
\end{proof}


\begin{lemma}[Coercivity]
	\label{coercive}
$J_{\lambda,\mu}$ is coercive and bounded from below on $\mathcal{N}_{\lambda,\mu}$ for all $\lambda>0$ and $\mu>0$.
\end{lemma}
\begin{proof}
Let $\lambda>0$ and $\mu>0$ and pick $(u,v)\in {\mathcal N}_{\lambda,\mu}$. Then, there holds
\begin{align*}
J_{\lambda,\mu}(u,v) &=\Big(\frac{1}{p}-\frac{1}{\alpha+\beta}\Big)\|(u,v)\|^p-\Big(\frac{1}{q}-\frac{1}{\alpha+\beta}\Big)
\int_{\Omega} (\lambda |u|^q+\mu |v|^q) dx \\
&\geq 
\Big(\frac{1}{p}-\frac{1}{\alpha+\beta}\Big)\|(u,v)\|^p-\Big(\frac{1}{q}-\frac{1}{\alpha+\beta}\Big)
S^{-\frac{q}{p}}|\Omega|^{\frac{\alpha+\beta-q}{\alpha+\beta}}\left(\lambda^{\frac{p}{p-q}}+\mu^{\frac{p}{p-q}}\right)^{\frac{p-q}{p}}\|(u,v)\|^q,
\end{align*}	
which yields the assertion.	
\end{proof}

\noindent
By Lemmas \ref{le1} and \ref{coercive}, for any $0<\lambda^{\frac{p}{p-q}}+\mu^{\frac{p}{p-q}}<\Lambda_1$, 
$$
\mathcal{N}_{\lambda,\mu}=\mathcal{N}_{\lambda,\mu}^+ \cup \mathcal{N}_{\lambda,\mu}^-
$$ 
and $J_{\lambda,\mu}$ is coercive and bounded from below on $\mathcal{N}_{\lambda,\mu}^+$ and $\mathcal{N}_{\lambda,\mu}^-$. Therefore we may define
$$
c_{\lambda,\mu} :=\inf\limits_{\mathcal{N}_{\lambda,\mu}}J_{\lambda,\mu},\qquad\,\,
c_{\lambda,\mu}^\pm :=\inf\limits_{\mathcal{N}_{\lambda,\mu}^\pm}J_{\lambda,\mu}.
$$
\noindent
Of course, by Lemma~\ref{coercive}, we have $c_{\lambda,\mu},c_{\lambda,\mu}^\pm>-\infty$. The following result is valid 

\begin{lemma}[$c_{\lambda,\mu}^+<0$ and $ c_{\lambda,\mu}^->0$]
	\label{lema}
	 Let $\Lambda_1$ be as in \eqref{lambda1}. Then the following facts hold
	\begin{itemize}
\item[$(i)$] If $0<\lambda^{\frac{p}{p-q}}+\mu^{\frac{p}{p-q}}<\Lambda_1,$ then $c_{\lambda,\mu}\leq c_{\lambda,\mu}^+<0$,
\item[$(ii)$] If $0<\lambda^{\frac{p}{p-q}}+\mu^{\frac{p}{p-q}}<(q/p)^{\frac{p}{p-q}}\Lambda_1,$ then $ c_{\lambda,\mu}^->d_0$ for some
$d_0=d_0(\lambda,\mu,p,q,n,s,|\Omega|)>0$.
\end{itemize}
\end{lemma}
\begin{proof}
	Let us prove $(i)$. Let $(u,v)\in \mathcal{N}_{\lambda,\mu}^+$. Then we have $\varphi_{u,v}''(1) >0$, which combined with \eqref{ma2} yields 
	\begin{eqnarray*}
		\frac{p-q}{2( \alpha+\beta -q)}\|(u,v)\|^p>\int_\Omega |u|^{\alpha}|v|^\beta dx.
	\end{eqnarray*}
	Therefore
	\begin{align*}
		J_{\lambda,\mu}(u,v)&=\Big(\frac{1}{p}-\frac{1}{q}\Big)\|(u,v)\|^p+2\Big(\frac{1}{q}-\frac{1}{\alpha+\beta} \Big)\int_\Omega |u|^{\alpha}|v|^\beta dx \\
		&<\Big[\Big(\frac{1}{p}-\frac{1}{q}\Big)+\Big(\frac{1}{q}-\frac{1}{\alpha+\beta}\Big)\frac{p-q}{ \alpha+\beta -q}\Big]\|(u,v)\|^p
		=-\frac{(p-q)(\alpha+\beta-p)}{pq(\alpha+\beta)}\|(u,v)\|^p<0.
	\end{align*}
	Therefore, $c_{\lambda,\mu}\leq c_{\lambda,\mu}^+<0$ follows from the definitions of $c_{\lambda,\mu}$ and $c_{\lambda,\mu}^+$.
	Let us now come to $(ii)$. Let $(u,v)\in \mathcal{N}_{\lambda,\mu}^-$. Then, we have $\varphi_{u,v}''(1)<0$, which combined with \eqref{ma2} yields 
	\begin{equation*}
		\frac{p-q}{2( \alpha+\beta -q)}\|(u,v)\|^p<\int_\Omega |u|^{\alpha}|v|^\beta dx.
	\end{equation*}
	By Young inequality and  the definition of $S$,  we obtain
	$$
	\int_\Omega |u|^{\alpha}|v|^\beta dx\leq   \frac{\alpha}{\alpha+\beta}\int_\Omega  |u|^{\alpha+\beta} dx+
	\frac{\beta}{\alpha+\beta}\int_\Omega  |v|^{\alpha+\beta} dx \leq S^{-\frac{\alpha+\beta}{p}}\|(u,v)\|^{\alpha+\beta}.
	$$
	Thus,
	\begin{equation*}
		\|(u,v)\|>\Big(\frac{p-q}{2( \alpha+\beta -q)}\Big)^{\frac{1}{\alpha+\beta-p}}S^{\frac{\alpha+\beta}{p(\alpha+\beta-p)}}.
	\end{equation*}
	Moreover, by H\"{o}lder inequality and   the definition of $S$, we find
	$$
	\int_\Omega (\lambda|u|^{q}+\mu|v|^q)dx\leq S^{-\frac{q}{p}}|\Omega|^{\frac{\alpha+\beta-q}{\alpha+\beta}}\left(\lambda^{\frac{p}{p-q}}+\mu^{\frac{p}{p-q}}\right)^{\frac{p-q}{p}}\|(u,v)\|^q.
	$$
	Therefore, if $0<\lambda^{\frac{p}{p-q}}+\mu^{\frac{p}{p-q}}<\big(\frac{q}{p}\big)^{\frac{p}{p-q}}\Lambda_1$, then we have
	\begin{align*}
		J_{\lambda,\mu}(u,v)&\geq\|(u,v)\|^q\Big[\Big(\frac{1}{p}-\frac{1}{\alpha+\beta}\Big)\|(u,v)\|^{p-q}- \Big(\frac{1}{q}-\frac{1}{\alpha+\beta}\Big)S^{-\frac{q}{p}}|\Omega|^{\frac{\alpha+\beta-q}{\alpha+\beta}}
		\Big(\lambda^{\frac{p}{p-q}}+\mu^{\frac{p}{p-q}}\Big)^{\frac{p-q}{p}}\Big]\nonumber\\
		&> \|(u,v)\|^q\Big[\Big(\frac{1}{p}-\frac{1}{\alpha+\beta}\Big)\Big(\frac{p-q}{2( \alpha+\beta -q)}\Big)^{\frac{p-q}{\alpha+\beta-p}}
		S^{\frac{(\alpha+\beta)(p-q)}{p(\alpha+\beta-p)}}- \nonumber\\
		&\qquad\qquad \Big(\frac{1}{q}-\frac{1}{\alpha+\beta}\Big)S^{-\frac{q}{p}}|\Omega|^{\frac{\alpha+\beta-q}{\alpha+\beta}}
		\Big(\lambda^{\frac{p}{p-q}}+\mu^{\frac{p}{p-q}}\Big)^{\frac{p-q}{p}}\Big]\geq d_0>0.
	\end{align*}
	This completes the proof.
\end{proof}

\subsection{The first solution}

 \noindent
 We now prove the existence of a first solution $(u_1,v_1)$ to \eqref{frac1}. First, we need some preliminary results.
 	
 	\begin{lemma}[Curves into $\mathcal{N}_{\lambda,\mu}$]
 		\label{lema1}
 		Let $\Lambda_1$ be as in \eqref{lambda1} and assume $0<\lambda^{\frac{p}{p-q}}+\mu^{\frac{p}{p-q}}<\Lambda_1$. Then
 		for any $z=(u,v)\in\mathcal{N}_{\lambda,\mu}$ there exists $\epsilon>0$ and a differentiable map 
 		$$
 		\xi: B(0,\epsilon)\subset E\rightarrow\mathbb{R}^+,
 		$$  
 		such that
 		$\xi(0)=1$ and $\xi(\omega)(z-\omega)\in\mathcal{N}_{\lambda,\mu}$ and
 		\begin{equation}\label{a41h}
 			\langle\xi'(0),\omega\rangle=- \frac{p\mathcal{A}(u,\omega_1)+p\mathcal{A}(v,\omega_2)-K_{\lambda,\mu}(z,\omega)-2\displaystyle\int_\Omega(\alpha|u|^{\alpha-2}u\omega_1|v|^\beta+\beta|u|^\alpha|v|^{\beta-2}v\omega_2)dx}
 			{(p-q)\|(u,v)\|^p-2(\alpha+\beta-q)
 				\displaystyle\int_{\Omega}|u|^{\alpha}|u|^{\beta}dx},
 		\end{equation}
 		for all $\omega=(\omega_1,\omega_2)\in E$, where
 		$$
 		K_{\lambda,\mu}(z,\omega)=q\int_\Omega(\lambda|u|^{q-2}u\omega_1+\mu|v|^{q-2}v\omega_2)dx.
 		$$
 	\end{lemma}
 	\begin{proof}
 		For $z=(u,v)\in\mathcal{N}_{\lambda,\mu}$, define a function $F_z: \mathbb{R}^+\times E\rightarrow\mathbb{R}$ by
 		\begin{align*}
 			F_z(\xi,\omega) &:=  \langle J'_{\lambda,\mu}(\xi(z-\omega)),\xi(z-\omega)\rangle\\
 			&=\xi^p\left(\mathcal{A}(u-\omega_1,u-\omega_1)+\mathcal{A}(v-\omega_2,v-\omega_2)\right)
 			-\xi^q \displaystyle\int_{\Omega}\left(\lambda|u-\omega_1|^{q}+\mu |v-\omega_2|^{q}\right)dx\nonumber\\
 			&-2\xi^{\alpha+\beta}\int_{\Omega}|u-\omega_1|^{\alpha}|v-\omega_2|^{\beta}dx,\quad\,\, \xi\in \mathbb{R}^+,\, \omega\in E.
 		\end{align*}
 		Then $F_z(1,0)=\langle J'_{\lambda,\mu}(z),z\rangle=0$ and, by Lemma \ref{le1}, we have
 		\begin{align*}
 			\frac{d}{d\xi} F_z(1,(0,0))&=p\|(u,v)\|^p-q\int_{\Omega}\left(\lambda|u|^{q}+\mu |v|^{q}\right)dx-
 			2(\alpha+\beta)\int_{\Omega}|u|^{\alpha}|v|^{\beta}dx \\
 			&= (p-q)\|(u,v)\|^p-2(\alpha+\beta-q)\int_{\Omega}|u|^{\alpha}|u|^{\beta}dx\not=0.
 		\end{align*}
 		By the Implicit Function Theorem there is $\epsilon>0$ and a $C^1$ map $\xi: B(0,\epsilon)\subset E\to \mathbb{R}^+$ with $\xi(0)=1$ and
 		\begin{equation*}\label{a41}
 			\langle\xi'(0),\omega\rangle=- \frac{p\mathcal{A}(u,\omega_1)+p\mathcal{A}(v,\omega_2)-K_{\lambda,\mu}(z,\omega)-2\displaystyle\int_\Omega(\alpha|u|^{\alpha-2}u\omega_1|v|^\beta+\beta|u|^\alpha|v|^{\beta-2}v\omega_2)dx}
 			{(p-q)\|(u,v)\|^p-2(\alpha+\beta-q)
 				\displaystyle\int_{\Omega}|u|^{\alpha}|u|^{\beta}dx},
 		\end{equation*}
 		and $F_z(\xi(\omega),\omega)=0$ for  all  $\omega\in B(0,\epsilon)$,
 		which  is equivalent to
 		$$
 		\langle J'_{\lambda,\mu}(\xi(\omega)(z-\omega)),\xi(\omega)(z-\omega)\rangle=0,
 		\quad\mbox{for\  all} \ \omega\in B(0,\epsilon),
 		$$
 		namely $\xi(\omega)(z-\omega)\in\mathcal{N}_{\lambda,\mu}.$
 	\end{proof}

 	\begin{lemma}[Curves into $\mathcal{N}^-_{\lambda,\mu}$]
 		\label{lema2}
 			Let $\Lambda_1$ be as in \eqref{lambda1} and assume $0<\lambda^{\frac{p}{p-q}}+\mu^{\frac{p}{p-q}}<\Lambda_1$. 
 		Then, for each $z\in\mathcal{N}^-_{\lambda,\mu}$, there is $\epsilon>0$ and a differentiable map
 		$$
 		\xi^-: B(0,\epsilon)\subset E\rightarrow\mathbb{R}^+
 		$$  
 		such that
 		$\xi^-(0)=1,$  $\xi^-(\omega)(z-\omega)\in\mathcal{N}^{-}_{\lambda,\mu}$ 
 		and
 		\begin{equation*}
 			\langle(\xi^-)'(0),\omega\rangle=- \frac{p\mathcal{A}(u,\omega_1)+p\mathcal{A}(v,\omega_2)-K_{\lambda,\mu}(z,\omega)-2\displaystyle\int_\Omega(\alpha|u|^{\alpha-2}u\omega_1|v|^\beta+\beta|u|^\alpha|v|^{\beta-2}v\omega_2)dx}
 			{(p-q)\|(u,v)\|^p-2(\alpha+\beta-q)
 				\displaystyle\int_{\Omega}|u|^{\alpha}|u|^{\beta}dx},
 		\end{equation*}
 		for every $\omega\in B(0;\epsilon)$.
 	\end{lemma}
 	\begin{proof}
 		Arguing as in the proof of Lemma \ref{lema1}, there is
 		$\epsilon>0$ and a differentiable map $\xi^-: B(0,\epsilon)\subset
 		E\rightarrow\mathbb{R}^+$ such that $\xi^-(0)=1,$
 		$\xi^-(\omega)(z-\omega)\in\mathcal{N}_{\lambda,\mu}$ for all $\omega\in B(0,\epsilon)$
 		and (\ref{a41h}). 
 		Since
 		\begin{equation*} 
 			\varphi_{u,v}''(1)
 			=(p-q)\|(u,v)\|^p-2(\alpha+\beta-q)\int_\Omega |u|^{\alpha}|v|^\beta dx<0,
 		\end{equation*}
 		by continuity we have
 		\begin{align*} 
 			\varphi_{\xi^-(\omega)(u-\omega_1),\xi^-(\omega)(v-\omega_2)}''(1)
 			&=(p-q)\|(\xi^-(\omega)(u-\omega_1),\xi^-(\omega)(v-\omega_2))\|^p\nonumber\\
 			&-
 			2((\alpha+\beta)-q)\int_\Omega |\xi^-(\omega)(u-\omega_1)|^{\alpha}|\xi^-(\omega)(v-\omega_2)|^\beta dx<0,
 		\end{align*}
 		if $\epsilon$ is sufficiently small, which implies $\xi^-(\omega)(z-\omega)\in\mathcal{N}^{-}_{\lambda,\mu}$.
 	\end{proof}

 	\begin{proposition}[$(PS)_{c_{\lambda,\mu}}$-sequences]
 		\label{Pss}
 		The following facts hold: \\
 		(i) If $0<\lambda^{\frac{p}{p-q}}+\mu^{\frac{p}{p-q}}<\Lambda_1$, then there is a $(PS)_{c_{\lambda,\mu}}$-sequence $\{(u_k,v_k)\}\subset\mathcal{N}_{\lambda,\mu}$ for $J_{\lambda,\mu}$;\\
 		(ii) If $0<\lambda^{\frac{p}{p-q}}+\mu^{\frac{p}{p-q}}<(q/p)^{\frac{p}{p-q}}\Lambda_1$, there is a $(PS)_{c_{\lambda,\mu}^-}$-sequence
 		$\{(u_k,v_k)\}\subset\mathcal{N}_{\lambda,\mu}^-$ for $J_{\lambda,\mu}$.
 	\end{proposition}
 	\begin{proof}
 		(i) By Ekeland Variational Principle, there exists a minimizing sequence $\{(u_k,v_k)\}\subset\mathcal{N}_{\lambda,\mu}$ such that
 		\begin{equation}\label{r1}
 			J_{\lambda,\mu}(u_k,v_k)  <c_{\lambda,\mu}+\frac{1}{k},   \qquad
 			J_{\lambda,\mu}(u_k,v_k)<J_{\lambda,\mu}(w_1,w_2)+\frac{1}{k}\|(w_1,w_2)-(u_k,v_k)\|,
 		\end{equation}
 		for each $(w_1,w_2)\in\mathcal{N}_{\lambda,\mu}$.
 		Taking $k$ large and using $c_{\lambda,\mu}<0$, we have
 		\begin{equation}\label{r2}
 			J_{\lambda,\mu}(u_k,v_k)= \Big(\frac{1}{p}-\frac{1}{\alpha+\beta}\Big)\|(u_k,v_k)\|^p
 			-\Big(\frac{1}{q}-\frac{1}{\alpha+\beta}\Big)\int_{\Omega}(\lambda|u_{k}|^q+\mu|v_{k}|^q)dx
 			<\frac{c_{\lambda,\mu}}{2}.   
 		\end{equation}
 		This yields that
 		\begin{equation}\label{r3}
 			-\frac{q(\alpha+\beta)}{2(\alpha+\beta-q)}c_{\lambda,\mu}<\int_{\Omega}(\lambda|u_{k}|^q+\mu|v_{k}|^q)dx
 			\leq S^{-\frac{q}{p}}|\Omega|^{\frac{\alpha+\beta-q}{\alpha+\beta}}\left(\lambda^{\frac{p}{p-q}}+\mu^{\frac{p}{p-q}}\right)^{\frac{p-q}{p}}\|(u_k,v_k)\|^q.   
 		\end{equation}
 		Consequently, $(u_k,v_k)\not=0$ and combining with
 		\eqref{r2} and \eqref{r3} and using H\"{o}lder inequality
 		\begin{align}
 		& \|(u_k,v_k)\|>\Big[-\frac{q(\alpha+\beta)}{2(\alpha+\beta-q)}c_{\lambda,\mu}S^{\frac{q}{p}}|\Omega|^{-\frac{\alpha+\beta-q}{\alpha+\beta}}\left(\lambda^{\frac{p}{p-q}}+\mu^{\frac{p}{p-q}}\right)^{\frac{q-p}{p}}\Big]^{\frac{1}{q}}, \notag\\
 	&
 		\label{r4}
 			\|(u_k,v_k)\|<\Big[\frac{p(\alpha+\beta-q)}{q(\alpha+\beta-p)}S^{-\frac{q}{p}}|\Omega|^{\frac{\alpha+\beta-q}{\alpha+\beta}}\left(\lambda^{\frac{p}{p-q}}+\mu^{\frac{p}{p-q}}\right)^{\frac{p-q}{p}}\Big]^{\frac{1}{p-q}}.
 		\end{align}
 		Now we prove that $\|J'_{\lambda,\mu}(u_k,v_k)\|_{E^{-1}}\to 0$ as $k\rightarrow\infty$.
 		Fix $k\in{\mathbb N}$. By applying Lemma  \ref{lema1} to $z_k=(u_k,v_k)$, we obtain the function $\xi_k:
 		B(0,\epsilon_k)\rightarrow{\mathbb R}^+$ for some $\epsilon_k>0,$ such that
 		$\xi_k(h)(z_k-h)\in\mathcal{N}_{\lambda,\mu}.$ Take $0<\rho<\epsilon_k.$
 		Let $w\in E$ with $w\not\equiv0$ and put $h^*=\frac{\rho w}{\|w\|}.$
 		We set  $h_{\rho}=\xi_k(h^*)(z_k-h^*)$. Then
 		$h_{\rho}\in\mathcal{N}_{\lambda,\mu},$ and we have from (\ref{r1}) 
 		$$
 		J_{\lambda,\mu}(h_\rho)-J_{\lambda,\mu}(z_k)\geq-\frac{1}{k}\|h_\rho-z_k\|.
 		$$
 		By the Mean Value Theorem, we get
 		$$\langle J'_{\lambda,\mu}(z_k),h_\rho-z_k\rangle+o(\|h_\rho-z_k\|)\geq-\frac{1}{k}\|h_\rho-z_k\|.
 		$$
 		Thus, we have
 		\begin{equation*}
 			\langle J'_{\lambda,\mu}(z_k),-h^*\rangle+(\xi_k(h^*)-1)\langle J'_{\lambda,\mu}(z_k),z_k-h^*\rangle\geq-\frac{1}{k}\|h_\rho-z_k\|+o(\|h_\rho-z_k\|).
 		\end{equation*}
 		Whence, from $\xi_k(h^*)(z_k-h^*)\in\mathcal{N}_{\lambda,\mu},$ it follows that
 		\begin{equation*}
 			-\rho\big\langle J'_{\lambda,\mu}(z_k),\frac{w}{\|w\|}\big\rangle+(\xi_k(h^*)-1)
 			\langle J'_{\lambda,\mu}(z_k)-J'_{\lambda,\mu}(h_\rho),z_k-h^*\rangle  
 			\geq-\frac{1}{k}\|h_\rho-z_k\|+o(\|h_\rho-z_k\|).
 		\end{equation*}
 		Hence, we get
 		\begin{equation}\label{r5}
 			\big\langle J'_{\lambda,\mu}(z_k),\frac{w}{\|w\|}\big\rangle\leq \frac{1}{k\rho}\|h_\rho-z_k\|
 			+\frac{o(\|h_\rho-z_k\|)}{\rho} + \frac{(\xi_k(h^*)-1)}{\rho}\langle J'_{\lambda,\mu}(z_k)-J'_{\lambda,\mu}(h_\rho),z_k-h^*\rangle.  \notag
 		\end{equation}
 		Since $\|h_\rho-z_k\|\leq\rho|\xi_k(h^*)|+|\xi_k(h^*)-1|\|z_k\|$ and
 		$$
 		\lim_{\rho\rightarrow0}\frac{|\xi_k(h^*)-1|}{\rho}\leq\|\xi_k'(0)\|.
 		$$
 		Fixed $k\in{\mathbb N}$, if $\rho\rightarrow0$ in \eqref{r5}, then by virtue of
 		\eqref{r4} we can choose $C>0$ independent of $\rho$ such that
 		\begin{equation*}
 			\Big\langle
 			J'_{\lambda,\mu}(z_k),\frac{w}{\|w\|}\Big\rangle\leq\frac{C}{k}(1+\|\xi_k'(0)\|).
 		\end{equation*}
 		Thus, we are done if $\sup_{k\in {\mathbb N}}\|\xi_k'(0)\|_{E^*}<\infty$.
 		By \eqref{a41h}, \eqref{r4} and  H\"{o}lder inequality, we have
 		$$
 		\big|\langle\xi_k'(0),h\rangle\big|\leq\frac{C_1\|h\|}{\Big|(p-q)\|(u_k,v_k)\|^p-2(\alpha+\beta-q)
 			\displaystyle\int_{\Omega}|u_{k}|^{\alpha}v_{k}|^{\beta}dx\Big|}
 		$$
 		for some $C_1>0.$ We only need to prove that
 		$$
 		\Big|{(p-q)\|(u_k,v_k)\|^p-2(\alpha+\beta-q)
 			\int_{\Omega}|u_{k}|^{\alpha}|v_{k}|^{\beta}dx}\Big|\geq C_2,
 		$$
 		for some $C_2>0$ and $k$ large. By contradiction, suppose there is a subsequence $\{(u_k,v_k)\}_{k\in{\mathbb N}}$ with
 		\begin{eqnarray}\label{r7a}
 			(p-q)\|(u_k,v_k)\|^p-2(\alpha+\beta-q)
 			\int_{\Omega}|u_{k}|^{\alpha}|v_{k}|^{\beta}dx=o_k(1).
 		\end{eqnarray}
 		By  \eqref{r7a} and the fact that $(u_k,v_k)\in\mathcal{N}_{\lambda,\mu}$, we have
 		\begin{align}\label{r7ap}
 &			\|(u_k,v_k)\|^p=\frac{2(\alpha+\beta-q)}{p-q}\int_{\Omega}|u_{k}|^{\alpha}|v_{k}|^{\beta}dx+o_k(1), \\
 		\label{r7ui}
 &			\|(u_k,v_k)\|^p=\frac{\alpha+\beta-q}{\alpha+\beta-p}\int_\Omega(\lambda|u_k|^q+\mu|v_k|^q)dx+o_k(1).
 		\end{align}
 		By Young inequality, it follows that $\int_\Omega  |u_k|^{\alpha}|v_k|^\beta dx\leq  S^{-\frac{\alpha+\beta}{p}}\|(u_k,v_k)\|^{\alpha+\beta}.$
 		By this and (\ref{r7ap}), we get
 		\begin{eqnarray}\label{frac5a}
 			\|(u_k,v_k)\|\geq\left(\frac{p-q}{2(\alpha+\beta-q)}S^{\frac{\alpha+\beta}{p}}\right)^{\frac{1}{\alpha+\beta-p}}+o_k(1),
 		\end{eqnarray}
 		Moreover, from \eqref{r7ui}, and by H\"{o}lder inequality, we obtain
 		\begin{equation*}
 			\|(u_k,v_k)\|^p\leq \frac{\alpha+\beta-q}{\alpha+\beta-p}  |\Omega|^{\frac{\alpha+\beta-q}{\alpha+\beta}}S^{-\frac{q}{p}}
 			\left(\lambda^{\frac{p}{p-q}}+\mu^{\frac{p}{p-q}}\right)^{\frac{p-q}{p}}  \|(u_k,v_k)\|^q+o_k(1).
 		\end{equation*}
 		Thus
 		\begin{equation}\label{frac5b}
 			\|(u_k,v_k)\|\leq\left(\frac{\alpha+\beta-q}{\alpha+\beta-p}S^{-\frac{q}{p}}|\Omega|^{\frac{\alpha+\beta-q}{\alpha+\beta}}\right)^{\frac{1}{p-q}}
 			\left(\lambda^{\frac{p}{p-q}}+\mu^{\frac{p}{p-q}}\right)^{\frac{1}{p}}+o_k(1).
 		\end{equation}
 		From (\ref{frac5a}) and (\ref{frac5b}), and for $k$ large enough, we get
 		\begin{eqnarray*}\label{frac5}
 			\lambda^{\frac{p}{p-q}}+\mu^{\frac{p}{p-q}}\geq\left(\frac{p-q}{2(\alpha+\beta-q)} \right)^{\frac{p}{\alpha+\beta-p}}
 			\left(\frac{ \alpha+\beta-q}{\alpha+\beta-p}|\Omega|^{\frac{\alpha+\beta-q}{\alpha+\beta}}\right)^{-\frac{p}{p-q}}S^{\frac{\alpha+\beta}{\alpha+\beta-p}+\frac{q}{p-q}}
 			=\Lambda_1.
 		\end{eqnarray*}
 		which contradicts $0<\lambda^{\frac{p}{p-q}}+\mu^{\frac{p}{p-q}}<\Lambda_1$. Therefore,
 		$$
 		\big\langle J'_{\lambda,\mu}(u_k,v_k),\|w\|^{-1}w\big\rangle\leq\frac{C}{k}.
 		$$
 		This proves (i). By Lemma \ref{lema2}, using the same argument we can get (ii).
 	\end{proof}

 	\noindent
Here is the main result of the section.

\begin{proposition}[Existence of the first solution]
	\label{subcritical1}
 Let $\Lambda_1$ be as in \eqref{lambda1}. Assume that $0<\lambda^{\frac{p}{p-q}}+\mu^{\frac{p}{p-q}}<\Lambda_1$. Then 
there exists $(u_1,v_1)\in \mathcal{N}_{\lambda,\mu}^+$ such that 

(1) $J_{\lambda,\mu}(u_1,v_1)= c_{\lambda,\mu}=c_{\lambda,\mu}^+<0$;

(2) $(u_1,v_1)$ is a solution of problem (\ref{frac1}).
\end{proposition}

\begin{proof}
By (i) of Proposition~\ref{Pss}, there is a bounded minimizing sequence $\{(u_k,v_k)\}\subset \mathcal{N}_{\lambda,\mu}$ such that
\begin{equation*}
\lim\limits_{k\to\infty}J_{\lambda,\mu}(u_k,v_k)=c_{\lambda,\mu}\leq c_{\lambda,\mu}^+<0, \qquad J'_{\lambda,\mu}(u_k,v_k)=o_k(1)\ \ \mbox{in}\ \ E^\ast.
\end{equation*}
Then there exists $(u_1,v_1)\in E$ such that, up to a subsequence, 
$u_k\rightharpoonup u_1$, $v_k\rightharpoonup v_1$ in $X_0$ as well as 
$u_k\to u_1$ and $ v_k\to v_1$ strongly in $L^r(\Omega)$ for any $1\leq  r<p^\ast$.
Then, the Dominated Convergence Theorem yields
	$$
	\int_\Omega(\lambda|u_k|^{q}+\mu|v_k|^q)dx\to \int_\Omega(\lambda|u_1|^{q}+\mu|v_1|^q)dx,\quad \mbox{as}\ \ k\to\infty.
	$$
	It is easy to get that $(u_1,v_1)$ is a weak solution of \eqref{frac1}, cf.\ Lemma~\ref{ps2}.
	Now, since $(u_k,v_k)\in \mathcal{N}_{\lambda,\mu}$, we have
	\begin{align*}
		J_{\lambda,\mu}(u_k,v_k)&=\frac{\alpha+\beta-p}{p(\alpha+\beta)}\|(u_k,v_k)\|^p
		-\frac{\alpha+\beta-q}{q(\alpha+\beta)}\int_\Omega (\lambda |u_k|^{q}+\mu |v_k|^{q})dx\nonumber\\
		&\geq -\frac{\alpha+\beta-q}{q(\alpha+\beta)}\int_\Omega (\lambda |u_k|^{q}+\mu |v_k|^{q})dx.
	\end{align*}
	Then, from $c_{\lambda,\mu}<0$, we get
	$$
	\int_\Omega (\lambda |u_1|^{q}+\mu |v_1|^{q})dx\geq -\frac{q(\alpha+\beta)}{\alpha+\beta-q}c_{\lambda,\mu}>0
	$$
	Therefore, 
	$(u_1,v_1)\in \mathcal{N}_{\lambda,\mu}$ is a nontrivial solution of \eqref{frac1}.
	Next, we show that $(u_k,v_k)\to (u_1,v_1)$ strongly in $E$ and $J_{\lambda,\mu}(u_1,v_1)=c_{\lambda,\mu}^+$. In fact, since  
	$(u_1,v_1)\in \mathcal{N}_{\lambda,\mu}$, in light of Fatou's Lemma we get
	\begin{align*}
		c_{\lambda,\mu}\leq J_{\lambda,\mu}(u_1,v_1)&= \frac{\alpha+\beta-p}{p(\alpha+\beta)}\|(u_1,v_1)\|^p
		-\frac{\alpha+\beta-q}{q(\alpha+\beta)}\int_\Omega (\lambda |u_1|^{q}+\mu |v_1|^{q})dx\nonumber\\
		&\leq \liminf\limits_{k\to\infty}\left(\frac{\alpha+\beta-p}{p(\alpha+\beta)}\|(u_k,v_k)\|^p
		-\frac{\alpha+\beta-q}{q(\alpha+\beta)}\int_\Omega (\lambda |u_k|^{q}+\mu |v_k|^{q})dx\right)\nonumber\\
		&= \liminf\limits_{k\to\infty}J_{\lambda,\mu}(u_k,v_k)=c_{\lambda,\mu}.
	\end{align*}
	This implies that $J_{\lambda,\mu}(u_1,v_1)=c_{\lambda,\mu}$ and $\|(u_k,v_k)\|^p\to \|(u_1,v_1)\|^p$. 
	We also have
	$$
	\|(u_k-u_1,v_k-v_1)\|^p=\|(u_k,v_k)\|^p-\|(u_1,v_1)\|^p+o_k(1).
	$$
	Therefore $(u_k,v_k)\to (u_1,v_1)$ strongly in $E$. We claim that $(u_1,v_1)\in \mathcal{N}_{\lambda,\mu}^+$, 
	which yields $c_{\lambda,\mu}=c_{\lambda,\mu}^+$. Assume by contradiction that  $(u_1,v_1)\in \mathcal{N}_{\lambda,\mu}^-$. By Lemma \ref{firb2}, there exist unique $t_2>t_1>0$ such that 
	$$
	(t_1u_1,t_1v_1)\in\mathcal{N}_{\lambda,\mu}^+,\qquad 
	(t_2u_1,t_2v_1)\in\mathcal{N}_{\lambda,\mu}^-. 
	$$
	In particular, we have $t_1<t_2=1$. Since 
	$$
	\frac{d}{dt}J_{\lambda,\mu}(t_1u_1,t_1v_1)=0,\quad \frac{d^2}{dt^2}J_{\lambda,\mu}(t_1u_1,t_1v_1)>0,
	$$
	there exists $t^\ast\in(t_1,1]$ such that $J_{\lambda,\mu}(t_1u_1,t_1v_1)<J_{\lambda,\mu}(t^\ast u_1,t^\ast v_1)$.
	Then
	$$
	c_{\lambda,\mu}\leq J_{\lambda,\mu}(t_1u_1,t_1v_1)< J_{\lambda,\mu}(t^\ast u_1,t^\ast v_1)\leq J_{\lambda,\mu}( u_1,v_1)=c_{\lambda,\mu}
	$$
	which is a contradiction. Hence $(u_1,v_1)\in \mathcal{N}_{\lambda,\mu}^+$.
\end{proof}

\subsection{The second solution}
We next establish the existence of a minimum for $J_{\lambda,\mu}|_{\mathcal{N}_{\lambda,\mu}^-}$.

Let $S$ be as in (\ref{criticalfrac}). From  \cite{brasco}, we know that for $1<p<\infty$, $s\in(0,1)$, $n>ps$, there exists a minimizer for $S$, and for every minimizer $U$, there exist $x_0\in\mathbb{R}^n$ and a constant sign monotone function $u:\mathbb{R}\to\mathbb{R}$ such that $U(x)=u(|x-x_0|)$. In the following, we shall fix a radially symmetric nonnegative decreasing minimizer $U = U (r)$ for $S$. Multiplying $U$ by a positive constant if necessary, we may assume that
\begin{eqnarray}\label{careq}
(-\Delta)_p^sU=U^{p_s^\ast-1}\quad \mbox{in} \ \mathbb{R}^n.
\end{eqnarray}
For any $\eps>0$, we note that the function
\begin{equation*}
U_\eps(x)=\frac{1}{\eps^{\frac{n-ps}{p}}}U\left(\frac{|x|}{\eps}\right)
\end{equation*}
is also a minimizer for $S$ satisfying \eqref{careq}. In \cite{brasco}, the following asymptotic estimates for $U$ was provided.

\begin{lemma}[Optimal decay]
	\label{leesti}
There exist $c_1,c_2>0$ and $\theta>1$ such that for all $r>1$, 
$$
\frac{c_1}{r^{{\frac{n-ps}{p-1}}}}\leq U(r)\leq \frac{c_2}{r^{{\frac{n-ps}{p-1}}}},\qquad
\frac{U(\theta r)}{U(r)}\leq \frac{1}{2}.
$$
\end{lemma}

\noindent
Assume, without loss of generality, that $0\in\Omega$. For $\eps, \delta>0$, let
$$
m_{\eps,\delta}=\frac{U_\eps(\delta)}{U_\eps(\delta)-U_\eps(\theta\delta)},
$$
let
\begin{eqnarray*}
  g_{\eps,\delta}(t)= 
  \left\{ \arraycolsep=1.5pt
\begin{array}{lll}
0,\ \  \ &
{\rm if}\ 0\leq t\leq U_\eps(\theta\delta);\\[2mm]
m_{\eps,\delta}^p(t-U_\eps(\theta\delta)),\ \  \ &
{\rm if}\  U_\eps(\theta\delta)\leq t\leq U_\eps(\delta);\\[2mm]
t+U_\eps(\delta)(m_{\eps,\delta}^{p-1}-1), \ \ \quad & {\rm if}\ t\geq U_\eps(\delta),
\end{array}
\right.
\end{eqnarray*}
and
\begin{eqnarray*}
G_{\eps,\delta}(t)=\int_0^t  g'_{\eps,\delta}(\tau)^{\frac{1}{p}}d\tau= 
  \left\{ \arraycolsep=1.5pt
\begin{array}{lll}
0,\ \  \ &
{\rm if}\ 0\leq t\leq U_\eps(\theta\delta);\\[2mm]
m_{\eps,\delta} (t-U_\eps(\theta\delta)),\ \  \ &
{\rm if}\   U_\eps(\theta\delta)\leq t\leq U_\eps(\delta);\\[2mm]
t , \ \ \quad & {\rm if}\ t\geq U_\eps(\delta).
\end{array}
\right.
\end{eqnarray*}
The functions $g_{\eps,\delta}$ and $G_{\eps,\delta}$ are nondecreasing and absolutely continuous. Consider the radially symmetric nonincreasing function
\begin{eqnarray}\label{uepsdeltadef}
u_{\eps,\delta}(r)=  G_{\eps,\delta}(U_\eps(r)),
\end{eqnarray}
which satisfies
\begin{eqnarray*}
u_{\eps,\delta}(r)=  \left\{ \arraycolsep=1.5pt
\begin{array}{ll}
U_\eps(r),\ \  \ &
{\rm if}\ r\leq   \delta,\\[1mm]
0,\ \  \ &{\rm if}\   r\geq\theta\delta.
\end{array}
\right.
\end{eqnarray*}

\noindent
We have the following estimates for $u_{\eps,\delta}$, which were proved in \cite[Lemma 2.7]{mpsy}.
\begin{lemma}[Norm estimates]
	\label{lebn}
There exists a constant $C=C(n,p,s)>0$ such that for any $0<\eps\leq\frac{\delta}{2} $, then the following estimates hold.
\begin{eqnarray*}
\int_{\mathbb{R}^{2n}}\frac{|u_{\eps,\delta}(x)-u_{\eps,\delta}(y)|^p}{|x-y|^{n+ps}}dxdy\leq S^{\frac{n}{ps}}+\O\big(\big(\frac{\eps}{\delta}\big)^{\frac{n-ps}{p-1}}\big),
\end{eqnarray*}
and
\begin{eqnarray*}
\int_{\mathbb{R}^n}|u_{\eps,\delta}(x)|^{p_s^\ast}dx\geq S^{\frac{n}{ps}}-C\big(\big(\frac{\eps}{\delta}\big)^{\frac{n}{p-1}}\big).
\end{eqnarray*}
\end{lemma}

\noindent
Next, an important technical lemma. This is the only point where
we use conditions \eqref{conditions-main} on $p,s,q,n$.

\begin{lemma}[$c_{\lambda,\mu}^-<c_\infty$]
	\label{psc}
Assume that \eqref{conditions-main} hold.
	Then there exists $\Lambda_2>0$ such that, for  
	$$
	0<\lambda^{\frac{p}{p-q}}+\mu^{\frac{p}{p-q}}<\Lambda_2,
	$$ 
	there exists $(u,v)\in E\backslash\{(0,0)\}$ with $u\geq 0,v\geq 0$, such that
	\begin{equation*}
	\label{psc1}
		\sup\limits_{t\geq0}J_{\lambda,\mu}(tu,tv)<c_\infty,
	\end{equation*}
	where $c_\infty$ is the constant given in \eqref{pscv}. In particular $c_{\lambda,\mu}^-<c_\infty$, for all $0<\lambda^{\frac{p}{p-q}}+\mu^{\frac{p}{p-q}}<\Lambda_2$.
\end{lemma}

\begin{proof}
	Write $J_{\lambda,\mu}(u,v)=J(u,v)-K(u,v)$
	where the functions
	$J: E\to\mathbb{R}$ and $K: E\to\mathbb{R}$ are defined by
	\begin{eqnarray*}\label{funj}
		J(u,v)=\frac{1}{p}\|(u,v)\|^p-\frac{2}{\alpha+\beta}\int_\Omega|u|^\alpha|v|^\beta dx,
		\qquad
		K(u,v)= \frac{1}{q}\int_\Omega(\lambda|u|^q+\mu|v|^q)dx.
	\end{eqnarray*}
	Set $u_0:=  \alpha^{\frac{1}{p}} u_{\eps,\delta}$, $v_0:= \beta^{\frac{1}{p}} u_{\eps,\delta}$, where $u_{\eps,\delta}$ is defined by (\ref{uepsdeltadef}). 
	The map $h(t):=J(tu_0,tv_0)$ satisfies $h(0)=0$, $h(t)>0$ for $t>0$ small and $h(t)<0$ for $t>0$ large. Moreover, $h$ maximizes at point  
	$$
	t_\ast:=\Bigg(\frac{\|(u_0,v_0)\|^p}{2\displaystyle\int_\Omega|u_0|^\alpha|v_0|^\beta dx}\Bigg)^{\frac{1}{\alpha+\beta-p}}.
	$$
	Thus, we have
	\begin{align*} 
		\sup\limits_{t\geq0}J(tu_0,tv_0)&=h(t_\ast)=\frac{t_\ast^p}{p}\|(u_0,v_0)\|^p-\frac{2t_\ast^{\alpha+\beta}}{\alpha+\beta}\int_\Omega|u_0|^\alpha|v_0|^\beta dx\nonumber\\
		&=\left(\frac{1}{p}-\frac{1}{\alpha+\beta}\right)\frac{\|(u_0,v_0)\|^{\frac{p(\alpha+\beta)}{\alpha+\beta-p}}}{\left(2\displaystyle\int_\Omega|u_0|^\alpha|v_0|^\beta dx\right)^{\frac{p}{\alpha+\beta-p}}} \nonumber\\
		&=\left(\frac{1}{p}-\frac{1}{\alpha+\beta}\right)\frac{(\alpha+\beta)^{\frac{\alpha+\beta}{\alpha+\beta-p}}}{2^{\frac{p}{\alpha+\beta-p}}
			\alpha^{\frac{\alpha}{\alpha+\beta-p}}\beta^{\frac{\beta}{\alpha+\beta-p}}}\frac{\|u_{\eps,\delta}\|_{X_0}^{\frac{p(\alpha+\beta)}{\alpha+\beta-p}}}{\Big( \displaystyle\int_\Omega|u_{\eps,\delta}|^{\alpha+\beta} dx\Big)^{\frac{p}{\alpha+\beta-p}}} \nonumber\\
		&=\frac{s}{n}\frac{1}{2^{\frac{n-ps}{ps}}}
		\left[\Big(\frac{\alpha}{\beta}\Big)^{\frac{\beta}{\alpha+\beta}}
		+\Big(\frac{\beta}{\alpha}\Big)^{\frac{\alpha}{\alpha+\beta}}\right]^{\frac{n}{ps}}
		\Bigg[\frac{\|u_{\eps,\delta}\|_{X_0}^{p}}{\Big( \displaystyle\int_\Omega|u_{\eps,\delta}|^{p_s^\ast} dx\Big)^{\frac{p}{p_s^\ast}}}\Bigg]^{\frac{n}{ps}} .
	\end{align*}
	From Lemma \ref{lebn} and \eqref{alphabeta}, we have
	\begin{align}\label{funjasy}
		\sup\limits_{t\geq0}J(tu_0,tv_0)&\leq 
		\frac{s}{n}\frac{1}{2^{\frac{n-ps}{ps}}}
		\left[\Big(\frac{\alpha}{\beta}\Big)^{\frac{\beta}{\alpha+\beta}}
		+\Big(\frac{\beta}{\alpha}\Big)^{\frac{\alpha}{\alpha+\beta}}\right]^{\frac{n}{ps}}
		\Bigg[\frac{S^{\frac{n}{ps}}+\O\big(\big(\frac{\eps}{\delta}\big)^{\frac{n-ps}{p-1}}\big)}
		{\big( S^{\frac{n}{ps}}-C\big(\big(\frac{\eps}{\delta}\big)^{\frac{n}{p-1}}\big)\big)^{\frac{p}{p_s^\ast}}}\Bigg]^{\frac{n}{ps}} \nonumber\\
		&\leq \frac{2s}{n}\Big(\frac{S_{\alpha,\beta}}{2}\Big)^{\frac{n}{ps}}+\O\Big(\Big(\frac{\eps}{\delta}\Big)^{\frac{n-ps}{p-1}}\Big).
	\end{align}
	Let $\delta_{1}>0$ be such that for all $0<\lambda^{\frac{p}{p-q}}+\mu^{\frac{p}{p-q}}<\delta_1$,
	$$
	c_\infty=\frac{2s}{n}\Big(\frac{S_{\alpha,\beta}}{2}\Big)^{\frac{n}{ps}}
	-C_0\left(\lambda^{\frac{p}{p-q}}+\mu^{\frac{p}{p-q}}\right)>0.
	$$
	We have
	$J_{\lambda,\mu}(tu_0,tv_0)\leq \frac{t^p}{p}\|(u_0,v_0)\|^p\leq Ct^p$ for $t\geq 0$ and $\lambda,\mu>0.$
	Thus, there exists $t_0\in(0,1)$ such that
	$$
	\sup\limits_{0\leq t\leq t_0}J_{\lambda,\mu}(tu_0,tv_0)<c_\infty,\qquad \mbox{for\ all}\ 0<\lambda^{\frac{p}{p-q}}+\mu^{\frac{p}{p-q}}<\delta_1.
	$$
	Since $\alpha,\beta>1$,   it follows from \eqref{uepsdeltadef} and \eqref{funjasy} that
	\begin{align*}
		\sup\limits_{t\geq t_0}J_{\lambda,\mu}(tu_0,tv_0) &= \sup\limits_{t\geq t_0}\left[J(tu_0,tv_0)-K(tu_0,tv_0)\right]  \\
		&\leq \frac{2s}{n}\left(\frac{S_{\alpha,\beta}}{2}\right)^{\frac{n}{ps}}+\O\left(\left(\frac{\eps}{\delta}\right)^{\frac{n-ps}{p-1}}\right)-\frac{t_0^q}{q}
		\left(\lambda\alpha^{\frac{q}{p}}+\mu\beta^{\frac{q}{p}}\right)\int_{B(0,\delta)}|u_{\eps,\delta}|^qdx   \\
		&\leq \frac{2s}{n}\left(\frac{S_{\alpha,\beta}}{2}\right)^{\frac{n}{ps}}+\O\left(\left(\frac{\eps}{\delta}\right)^{\frac{n-ps}{p-1}}\right)-\frac{t_0^q}{q}
		\left(\lambda +\mu \right)\int_{B(0,\delta)}|u_{\eps,\delta}|^qdx.
	\end{align*}
	Fix now $\delta>0$ sufficiently small that $B_{\theta\delta}(0)\Subset\Omega$ (we assume without loss of generality that $0\in\Omega$), so that ${\rm supp}(u_{\epsilon,\delta})\subset\Omega$, according
	to formula \eqref{uepsdeltadef}. By means of Lemma \ref{leesti}, for any $0<\eps\leq \frac{\delta}{2}$, we have
	\begin{align*}
		\int_{B(0,\delta)}|u_{\eps,\delta}(x)|^qdx&=\int_{B(0, \delta )}|U_{\eps}(x)|^qdx \\
		&= \eps^{n-\frac{n-ps}{p}q}\int_{B(0,\frac{\delta}{\eps})}|U(x)|^qdx  
		\geq \eps^{n-\frac{n-ps}{p}q}\omega_{n-1}\int_1^{\frac{\delta}{\eps}}U(r)^qr^{n-1}dr \\
		&\geq \eps^{n-\frac{n-ps}{p}q}\omega_{n-1}c_1^q\int_1^{\frac{\delta}{\eps}}r^{n-\frac{n-ps}{p-1}q-1}dr 
		\simeq C
			\begin{cases}
				 \eps^{n-\frac{n-ps}{p}q},\ \  \ &
				\text {if $q>\frac{n(p-1)}{n-ps}$,} \\[1mm]
				 \eps^{n-\frac{n-ps}{p}q}|\log\eps|,  & \text{if  $q=\frac{n(p-1)}{n-ps}$,} \\[1mm]
				 \eps^{\frac{(n-ps)q}{p(p-1)}},  & \text{if  $q<\frac{n(p-1)}{n-ps}$.}
			\end{cases}
	\end{align*}
Therefore, taking into account  conditions \eqref{conditions-main}, we have
\begin{align}\label{funjasftgyhchen}
\sup\limits_{t\geq t_0}J_{\lambda,\mu}(tu_0,tv_0) 
&\leq   \frac{2s}{n}\Big(\frac{S_{\alpha,\beta}}{2}\Big)^{\frac{n}{ps}}\nonumber\\
&+C\big( \eps^{\frac{n-ps}{p-1}}\big) -C(\lambda+\mu)
\begin{cases}
 \eps^{n-\frac{n-ps}{p}q},\ \  \ &
\text{if $q>\frac{n(p-1)}{n-ps}$},\\[1mm]
 \eps^{n-\frac{n-ps}{p}q}|\log\eps|,\ \  \ & \text{if $q=\frac{n(p-1)}{n-ps}$}.
\end{cases}
\end{align}
For $\varepsilon=\big(\lambda^{\frac{p}{p-q}}+\mu^{\frac{p}{p-q}}\big)^{ \frac{p-1}{n-ps}}\in(0,\frac{\delta}{2})$, we get	
\begin{align*}\label{funjasftgyh}
\sup\limits_{t\geq t_0}J_{\lambda,\mu}(tu_0,tv_0) &\leq  \frac{2s}{n}\Big(\frac{S_{\alpha,\beta}}{2}\Big)^{\frac{n}{ps}}+  C\big(\lambda^{\frac{p}{p-q}}+\mu^{\frac{p}{p-q}}\big)\nonumber\\
&- 
C\left(\lambda +\mu \right)
\begin{cases}
\big(\lambda^{\frac{p}{p-q}}+\mu^{\frac{p}{p-q}}\big)^{  \frac{p-1}{n-ps}(n-\frac{n-ps}{p}q)},\ \  \ &
\text{if $q>\frac{n(p-1)}{n-ps}$}, \\[1mm]
\big(\lambda^{\frac{p}{p-q}}+\mu^{\frac{p}{p-q}}\big)^{  \frac{n(p-1)}{p(n-ps)} }
\big|\log\big(\lambda^{\frac{p}{p-q}}+\mu^{\frac{p}{p-q}}\big)\big|,\ \  \ & \text{if $q=\frac{n(p-1)}{n-ps}$}.
\end{cases}
\end{align*}
	If $q>\frac{n(p-1)}{n-ps}$, we can choose $\delta_2>0$, for $0<\lambda^{\frac{p}{p-q}}+\mu^{\frac{p}{p-q}}<\delta_2$, such that
	\begin{eqnarray}\label{funjasftgyhio}
	C\big(\lambda^{\frac{p}{p-q}}+\mu^{\frac{p}{p-q}}\big)- 
	C\left(\lambda +\mu \right)\big(\lambda^{\frac{p}{p-q}}+\mu^{\frac{p}{p-q}}\big)^
	{  \frac{p-1}{n-ps}(n-\frac{n-ps}{p}q)} <-C_0  (\lambda^{\frac{p}{p-q}}+\mu^{\frac{p}{p-q}}),
	\end{eqnarray}
	where $C_0$ is the positive constant defined in \eqref{cizero}.
	In fact, \eqref{funjasftgyhio} holds if 
	\begin{equation} \label{conqps}
	1+\frac{p}{p-q}\frac{p-1}{n-ps}\Big(n-\frac{n-ps}{p}q\Big)<\frac{p}{p-q}
	\quad \Leftrightarrow\quad q>\frac{n(p-1)}{n-ps}
	\end{equation}
	If instead $q=\frac{n(p-1)}{n-ps}$ we can choose $\delta_3>0$, for 
	$0<\lambda^{\frac{p}{p-q}}+\mu^{\frac{p}{p-q}}<\delta_3$, such that
	\begin{equation*}
	C\big(\lambda^{\frac{p}{p-q}}+\mu^{\frac{p}{p-q}}\big)- 
	C\left(\lambda +\mu \right)\big(\lambda^{\frac{p}{p-q}}+\mu^{\frac{p}{p-q}}\big)^
	{\frac{n(p-1)}{p(n-ps)}}\big|\log\big(\lambda^{\frac{p}{p-q}}+\mu^{\frac{p}{p-q}}\big)\big| <-C_0  (\lambda^{\frac{p}{p-q}}+\mu^{\frac{p}{p-q}}),
	\end{equation*}
	as $|\log (\lambda^{\frac{p}{p-q}}+\mu^{\frac{p}{p-q}})|\to+\infty$ for $\lambda,\mu\to 0$ and 
	$\left(\lambda +\mu \right)\big(\lambda^{\frac{p}{p-q}}+\mu^{\frac{p}{p-q}}\big)^
	{\frac{n(p-1)}{p(n-ps)}}\simeq \big(\lambda^{\frac{p}{p-q}}+\mu^{\frac{p}{p-q}}\big).$
	Then, taking 
	$$
	\Lambda_2=\min\big\{\delta_1,\delta_2,\delta_3,(\delta/2)^{{\frac{n-ps}{p-1}}}\big\}>0, 
	$$
	then for all
	$0<\lambda^{\frac{p}{p-q}}+\mu^{\frac{p}{p-q}}<\Lambda_2$, we have
	\begin{equation}\label{esinfty}
		\sup\limits_{t\geq0}J_{\lambda,\mu}(tu,tv)<c_\infty.
	\end{equation}
	Finally, fix $\lambda,\mu>0$ with $0<\lambda^{\frac{p}{p-q}}+\mu^{\frac{p}{p-q}}<\Lambda_2$. 
	Since $(u_0,v_0)\neq (0,0)$, from Lemma \ref{firb2}
	and \eqref{esinfty}, there exists $t_2>0$ such that $(t_2u_0,t_2v_0)\in\mathcal{N}_{\lambda,\mu}^-$ and
	$$
	c_{\lambda,\mu}^-\leq J_{\lambda,\mu}(t_2u_0,t_2v_0)\leq\sup\limits_{t\geq0}J_{\lambda,\mu}(tu_0,tv_0)<c_\infty,
	$$
	for all $0<\lambda^{\frac{p}{p-q}}+\mu^{\frac{p}{p-q}}<\Lambda_2$. This concludes the proof.
\end{proof}

\begin{proposition}[Existence of the second solution]
	\label{subcritical2}
There exists a positive constant $\Lambda_3>0$, such that for 
$$
0<\lambda^{\frac{p}{p-q}}+\mu^{\frac{p}{p-q}}<\Lambda_3,
$$ 
the functional $J_{\lambda,\mu}$ has a minimizer $(u_2,v_2)$ in $\mathcal{N}_{\lambda,\mu}^-$ and satisfies

(1) $J_{\lambda,\mu}(u_2,v_2)=c_{\lambda,\mu}^-$,

(2) $(u_2,v_2)$ is a solution of problem \eqref{frac1}.
\end{proposition}

\begin{proof}
Let $\Lambda_2$ be as in Lemma \ref{psc} and set  $\Lambda_3:=\{\Lambda_2, \left(q/p\right)^{\frac{p}{p-q}}\Lambda_1\}.$
By means of $(ii)$ of Proposition \ref{Pss}, for all $0<\lambda^{p/(p-q)}+\mu^{p/(p-q)}<\Lambda_3$, there exists a bounded $(PS)_{c_{\lambda,\mu}^-}$ sequence $(\tilde{u}_k,\tilde{v}_k)\}\subset \mathcal{N}_{\lambda,\mu}^-$ for $J_{\lambda,\mu}$.
By the same argument in the proof of  Proposition \ref{subcritical1}, there exists $(u_2,v_2)\in E$ such that, up to a subsequence,
$\tilde{u}_k\to u_2$, $\tilde{v}_k\to v_2$ strongly in $E$ and $J_{\lambda,\mu}(u_2,v_2)=c_{\lambda,\mu}^-$. Moreover, $(u_2,v_2)$ is a solution of problem (\ref{frac1}).
Next we show that $(u_2,v_2)\in \mathcal{N}_{\lambda,\mu}^-$. In fact, since $(\tilde{u}_k,\tilde{v}_k)\in \mathcal{N}_{\lambda,\mu}^-$, we have
\begin{equation*} 
	\varphi_{\tilde{u}_k,\tilde{v}_k}''(1) = (p-q)\|(\tilde{u}_k,\tilde{v}_k)\|^p-2((\alpha+\beta)-q)\int_\Omega |\tilde{u}_k|^{\alpha}|\tilde{v}_k|^\beta dx <0.
\end{equation*}
Since $\tilde{u}_k\to u_2$, $\tilde{v}_k\to v_2$ strongly in $E$, passing to the limit  we obtain
\begin{equation*}
\varphi_{u_2,v_2}''(1) = (p-q)\|(u_2,v_2)\|^p-2((\alpha+\beta)-q)\int_\Omega |u_2|^{\alpha}|v_2|^\beta dx \leq 0.
\end{equation*}
Since $\mathcal{N}_{\lambda,\mu}^0=\emptyset$, we conclude that $\varphi_{u_2,v_2}''(1) <0$
namely $(u_2,v_2)\in \mathcal{N}_{\lambda,\mu}^-$.	
\end{proof}

\section{Proof of the result concluded}
\label{finalsect}
\noindent
Taking $\Lambda_\ast=\min\{\Lambda_1,\Lambda_2,\Lambda_3\}$, by 
Propositions \ref{subcritical1} and \ref{subcritical2}, we 
know that for all 
$0<\lambda^{p/(p-q)}+\mu^{p/(p-q)}<\Lambda_\ast,$ 
problem (\ref{frac1}) has two solutions $(u_1,v_1)\in\mathcal{N}_{\lambda,\mu}^+$ and $(u_2,v_2)\in\mathcal{N}_{\lambda,\mu}^-$ in $E$. Since $\mathcal{N}_{\lambda,\mu}^+\cap\mathcal{N}_{\lambda,\mu}^-=\emptyset$, then these two solutions are distinct. 
We next show that $(u_1,v_1)$ and $(u_2,v_2)$ are not semi-trivial. 
We know that
\begin{eqnarray}\label{semilb}
J_{\lambda,\mu}(u_1,v_1)<0 \quad\text{and}\quad  J_{\lambda,\mu}(u_2,v_2)>0.
\end{eqnarray}
We note that
if $(u,0)$ (or $(0,v)$) is a semi-trivial solution of problem \eqref{frac1}, then \eqref{frac1} reduces to
\begin{eqnarray}\label{semiu}
(-\Delta)_p^s u =\lambda |u|^{q-2}u\ \ \mbox{in}\ \Omega,\qquad
u=0  \ \  {\rm in}\ \mathbb{R}^n\setminus\Omega.
\end{eqnarray}
Then
\begin{equation}\label{semilc}
J_{\lambda,\mu}(u,0)=\frac{1}{p}\int_Q\frac{|u(x)-u(y)|^p}{|x-y|^{n+ps}} \,dx\,dy
 -\frac{ \lambda}{q}\int_\Omega  |u|^{q} dx
=-\frac{p-q}{pq}\|u\|_{X_0}^p<0.
\end{equation}
From \eqref{semilb} and \eqref{semilc}, we get that $(u_2,v_2)$ is {\em not} semi-trivial.
Now we prove that $(u_1,v_1)$ is not semi-trivial. Without loss of generality, we may assume that
$v_1\equiv0$. Then $u_1$ is a nontrivial solution of \eqref{semiu}, and
$$
\|(u_1,0)\|^p=\|u_1\|_{X_0}^p=\lambda\int_\Omega|u_1|^qdx>0.
$$
Moreover, we may choose $w\in X_0\backslash\{0\}$ such that
$$
\|(0,w)\|^p=\|w\|_{X_0}^p=\mu\int_\Omega|w|^qdx>0.
$$
By Lemma \ref{firb2} there exists a unique $0<t_1<t_{\maxx}(u_1,w)$ such that
$
(t_1 u_1,t_1w)\in\mathcal{N}_{\lambda,\mu}^+,
$
where
$$
t_{\maxx}(u_1,w)=\Bigg(\frac{(\alpha+\beta-q)\displaystyle\int_\Omega (\lambda|u_1|^{q}+\mu|w|^q) dx}{(\alpha+\beta-p) \|(u_1,w)\|^p}\Bigg)^{\frac{1}{p-q}}=
\Big(\frac{\alpha+\beta-q}{\alpha+\beta-p}\Big)^{\frac{1}{p-q}}>1.
$$
Furthermore,
$$
J_{\lambda,\mu}(t_1u_1,t_1w)=\inf\limits_{0\leq t\leq t_{\maxx}}J_{\lambda,\mu}(tu_1,tw).
$$
This together with $(u_1,0)\in\mathcal{N}_{\lambda,\mu}^+$ implies that
$$
c_{\lambda,\mu}^+\leq J_{\lambda,\mu}(t_1u_1,t_1w)\leq J_{\lambda,\mu}(u_1,w)<J_{\lambda,\mu}(u_1,0)=c_{\lambda,\mu}^+,
$$
which is a contradiction. Hence $(u_1,v_1)$ is {\em not} semi-trivial too.
The proof is now complete.  \qed

\bigskip

\enddocument